\documentclass[3p]{elsarticle}
\usepackage{xcolor}
\usepackage{amssymb}
\usepackage{amsmath}

\usepackage{color}
\usepackage{tikz-cd}
\usepackage{amsthm}
\usepackage{hyperref}
\usepackage{cleveref}
\usepackage{array,makecell}
\usepackage{multirow}
\usepackage{float}

\definecolor{dark_blue}{RGB}{46,87,144}

\usepackage[normalem]{ulem}

\newcommand\rev[1]{\textcolor{black}{#1}}

\newcommand{\be}{\begin{equation}}
\newcommand{\ee}{\end{equation}}

\newcommand\n[1]{{\|#1\|}}

\newcommand{\fa}{\mathsf{a}}
\newcommand{\fc}{\mathsf{c}}

\newcommand{\fm}{\mathsf{m}}
\newcommand{\fn}{\mathsf{n}}

\newcommand{\fA}{\mathsf{A}}
\newcommand{\fC}{\mathsf{C}}

\newcommand{\fI}{\mathsf{I}}
\newcommand{\fK}{\mathsf{K}}
\newcommand{\fQ}{\mathsf{Q}}
\newcommand{\fM}{\mathsf{M}}
\newcommand{\fN}{\mathsf{N}}
\newcommand{\fP}{\mathsf{P}}
\newcommand{\fT}{\mathsf{T}}

\newcommand{\mC}{\mathcal{C}}
\newcommand{\mF}{\mathcal{F}}
\newcommand{\mL}{\mathcal{L}}
\newcommand{\mO}{\mathcal{O}}
\newcommand{\mK}{\mathcal{K}}
\newcommand{\mV}{\mathcal{V}}

\newcommand{\bb}{{\bf b}}

\newcommand{\bd}{{\bf d}}
\newcommand{\br}{{\bf r}}
\newcommand{\bq}{{\bf q}}
\newcommand{\bu}{{\bf u}}
\newcommand{\bv}{{\bf v}}
\newcommand{\bx}{{\bf x}}
\newcommand{\bw}{{\bf w}}

\newcommand{\bA}{{\bf A}}
\newcommand{\bC}{{\bf C}}

\newcommand{\bH}{{\bf H}}
\newcommand{\bI}{{\bf I}}
\newcommand{\bP}{{\bf P}}
\newcommand{\bK}{{\bf K}}
\newcommand{\bM}{{\bf M}}
\newcommand{\bN}{{\bf N}}

\newcommand{\bQ}{{\bf Q}}

\newcommand{\IC}{\mathbb{C}}
\newcommand{\IR}{\mathbb{R}}

\newcommand{\CA}{\textup{(\!(CA)\!)}}
\newcommand{\A}{\textup{(\!(A)\!)}}
\newcommand{\KL}{\operatorname{K}_{\Lambda_h}}
\newcommand{\K}{\operatorname{K}}

\newcommand{\iii}[1]{{\left\vert\kern-0.25ex\left\vert\kern-0.25ex\left\vert #1 
    \right\vert\kern-0.25ex\right\vert\kern-0.25ex\right\vert}}

\usepackage{tikz}

\usetikzlibrary{shapes,arrows,positioning}

\usetikzlibrary{babel}
\usetikzlibrary{shapes,arrows}
\definecolor{dark_blue_pers}{RGB}{46,87,144}
\definecolor{blue_pers}{RGB}{54,104,171}
\definecolor{grey_pers}{RGB}{245,245,245}
\definecolor{red_pers}{RGB}{213,78,33}

\tikzstyle{block} = [rectangle, draw, fill=blue_mentor!30, 
    text width=6.5em, text centered, rounded corners, minimum height=4em]
\tikzstyle{objection} = [rectangle, draw, fill=red_mentor!30, 
    text width=6.5em, text centered, rounded corners, minimum height=4em]
\tikzstyle{line} = [draw, -latex']
\tikzstyle{cloud} = [draw, rectangle, rounded corners,fill=blue_pers!30, node distance=3cm,
    minimum height=2em]

\tikzstyle{back} = [rectangle, draw, fill=grey_pers, text width=12em, 
    text centered, rounded corners, minimum height=2.5em]
\tikzstyle{back_b} = [rectangle, draw, fill=blue_pers!30, 
    text width=12em, text centered, rounded corners, minimum height=2.5em]
\tikzstyle{back_g} = [rectangle, draw, fill=grey_pers!100, 
    text width=5em, text centered, rounded corners, minimum height=1.5em]

\newtheorem{lemma}{Lemma}
\newtheorem{proposition}{Proposition}
\newtheorem{remark}{Remark}
\newtheorem{theorem}{Theorem}
\newtheorem{definition}{Definition}
\newtheorem{corollary}{Corollary}

\newtheorem{assumption}{Assumption}

\makeatletter
\def\thmheadbrackets#1#2#3{%
  \thmname{#1}\thmnumber{\@ifnotempty{#1}{ }\@upn{#2}}%
  \thmnote{ {\the\thm@notefont#3}}}
\makeatother

\newtheoremstyle{defbrakets}
  {}
  {}
  {\normalfont}
  {}
  {\bfseries}
  {.}
  { }
  {\thmheadbrackets{#1}{#2}{#3}}

\theoremstyle{defbrakets}

\journal{}
\begin{document}

\begin{frontmatter}
\title{Bi-Parametric Operator Preconditioning}

\author[Catolica]{Paul Escapil-Inchausp\'e}
\author[AI]{Carlos Jerez-Hanckes}

\address[Catolica]{Pontificia Universidad Cat\'olica de Chile, School of Engineering, Santiago, Chile}
\address[AI]{Universidad Adolfo Ib\'a\~nez, Faculty of Engineering and Sciences, Santiago, Chile}
\begin{abstract}
We extend the operator preconditioning framework [R.~Hiptmair, \emph{Comput. Math. with Appl.} 52 (2006), pp.~699--706] to Petrov-Galerkin methods while accounting for parameter-dependent perturbations of both variational forms and their preconditioners, as occurs when performing numerical approximations. By considering different perturbation parameters for the original form and its preconditioner, our bi-parametric abstract setting leads to robust and controlled schemes. For Hilbert spaces, we derive exhaustive linear and super-linear convergence estimates for iterative solvers, such as $h$-independent convergence bounds, when preconditioning with low-accuracy or, equivalently, with highly compressed approximations.
\end{abstract}


\end{frontmatter}
\section{Introduction}\label{sec:Introduction}
Variational equations---\emph{continuous weak forms} \cite[Section 3.1]{betcke2020product}---in suitably defined reflexive Banach spaces $X$, $Y$, or equivalently \cite[Proposition A.21]{ern2004theory} as operator equations---\emph{continuous strong forms} \cite[Section 3.1]{betcke2020product}---have successfully been employed to model a plethora of phenomena, particularly in the form of integro-differential equations. In general, one can only approximate solutions by solving linear systems or \emph{matrix equations} arising from the continuous infinite-dimensional counterparts. Galerkin methods are a widely accepted choice to derive such linear systems due to their solid theoretical and practical understanding. Specifically, \emph{Petrov-Galerkin (PG) methods} provide a generic framework for operator equations with operators of the form $\fA:X\to Y'$, allowing to choose different trial and test spaces. Within PG methods, one finds \emph{Bubnov-Galerkin (BG) methods}, namely, the case when $\fA:X\to X'$ as well as PG for endomorphisms (PGE), i.e.,~$\fA:X\to X$, as in second-kind Fredholm integral equations, wherein $\fA$  is a compact perturbation of the identity in $X$.

Most relevant applications lead to large linear systems solved by iterative methods \cite{saad2003iterative} such as Krylov (subspace) methods \cite[Chapters 6 and 7]{saad1986gmres} as direct inversion quickly becomes computationally impractical. For real symmetric (resp.~complex Hermitian) positive definite matrices, the standard choice is the conjugate gradient method (CG) \cite{hestenes1952methods}, whereas the general minimal residual method (GMRES) and its $m$-restarted variant GMRES$(m)$ \cite{saad1986gmres} are common alternatives for nonsingular indefinite complex matrices. For these methods, convergence of the residual strongly depends on matrix properties inherited from the continuous (resp.~discrete) operator. Features such as the field-of-values (FoV) or singular values distributions are key to obtain residual convergence bounds \cite{sarkis2007optimal,beckermann2005some,steinbach2007numerical,nevanlinna2012convergence}. Yet, convergence for these methods can be slow, with performance commonly deteriorating as the linear system dimension increases. Thus, the need for robust preconditioning techniques.

For a linear system $\bA \bu = \bb$, in our case spawned by any PG method, preconditioning consists in the application of a (left) preconditioner $\bP$ such that
$$ 
\bP \bA \bu = \bP \bb.
$$
We say that the preconditioner $\bP$ is good if: (i) it is relatively cheap to compute; and (ii) the product $\bP \bA$ approximates the identity matrix or iterative solvers perform better than on the original linear system. In this note, we focus on the framework of operator preconditioning (OP). Successfully applied to BG methods \cite{hiptmair2006operator,christiansen2000preconditionneurs}---denoted OP-BG---, we aim at extending OP to general PG methods (OP-PG) as well as understanding the effects of numerical perturbations in iterative solvers.

Fundamentally, OP relies on finding suitable endomorphic operator equations, i.e.~mappings onto the same function spaces, leading to bounded \emph{spectral condition numbers}. In the BG setting, one has reflexive Banach spaces $X$, $V$ and an operator $\fA:X\to X'$, for which one considers another operator $\fC:V\to V'$, such that
\be\label{eq:OP_BG}
\text{(OP-BG)}\quad\begin{tikzcd}
X \arrow{r}{\fA }  & X' \arrow{d}{\fM^{-*}} \\
V' \arrow{u}{\fM^{-1}} & V\arrow{l}{\fC} 
\end{tikzcd},
\ee
with $\fM: X \to V'$ linking the domain of $\fA$ and the range of $\fC$. The preconditioning operator is then $\fP:= \fM^{-1}\fC\fM^{-*}$. Similarly, opposite-order OP has been considered for PG methods \cite{andreev2012stability}, particularly in the context of pseudo-differential operators \cite{steinbach1998construction,mclean1999boundary,hsiao2008boundary}, i.e.~for $\fA : X \to Y'$ and $\fC: Y'\to X$, with\footnote{Recently, \cite{stevenson2019uniform} proposed a construction with $\fM$, $\fN \neq \fI$, whose discretization leads to $\bM$ and $\bN$ being diagonal matrices.} $\fM=\fN=:\fI$, leading to
\be\label{eq:OP_OO}
\text{(opposite-order OP)}\quad\begin{tikzcd}[every arrow/.append style={shift left}]
X \arrow{r}{\fA }  & Y' \arrow{l}{\fC} 
\end{tikzcd}.
\ee
However, there is no known result for general OP-PG, which would encompass both OP-BG and opposite-order PG. This entails considering the following more general framework. For reflexive Banach spaces $X,$ $Y$, $V$ and $W$, and a preconditioner $\fC: V \to W'$ to $\fA: X \to Y'$, we need to build the commuting diagram:
\be\label{eq:OP_PG}
\text{(OP-PG)}\quad\begin{tikzcd}
X \arrow{r}{\fA }  & Y' \arrow{d}{\fN^{-1}} \\
W' \arrow{u}{\fM^{-1}} & V\arrow{l}{\fC} 
\end{tikzcd},
\ee
with $\fM: X \to W'$ and $\fN: V \to Y'$ linking the domain and range spaces for $\fA$ and $\fC$, and leading to an endomorphism on $X$. Notice that in this case $\fP:= \fM^{-1}\fC\fN^{-1}$ and that \eqref{eq:OP_PG} reduces to \eqref{eq:OP_BG} if $W=V$, $Y=X$ and $\fN = \fM^*$. In this regard, our main contribution is a theory for OP-PG for which we provide estimates for spectral and Euclidean condition numbers. For the latter, we make use of the synthesis operator linking the domain space and its basis expansion, thereby acknowledging the dimension dependence. 

Yet, and despite leading to bounded spectral condition numbers, OP does not necessarily ensure convergence for iterative solvers such as GMRES or GMRES$(m)$. Theoretically, one requires further assumptions on the induced problems, related primarily to the matrix FoV distribution \cite{starke1997field,liesen2012field}, to obtain linear convergence results for GMRES. Still, these bounds are pessimistic \cite{liesen2012field,kirby2010functional}, with convergence radius for GMRES close to one. This justifies the derivation of sharper convergence results at the expense of tighter assumptions on the operators. For instance, one can observe a super-linear convergence of the iterative scheme for systems derived from second-kind Fredholm operator equations. \cite{moret1997note,winther1980some,campbell1996gmres,axelsson2018superlinear}. 

Furthermore, though OP general properties are retained as the linear system dimension increases, it can quickly become impractical. A well-known example is the dual mesh-based OP---also known as multiplicative Calder\'on preconditioning---for boundary element methods \cite{hiptmair2006operator,steinbach1998construction,andriulli2008multiplicative}. Indeed, due to barycentric grid refinement, the standard method entails a dramatic increase in memory and computational costs. To counter this, low-accuracy Calder\'on preconditioners have been recently proposed with promising results \cite{bebendorf2008hierarchical,escapiljerez,fierro2020fast}. Indeed, iterative solvers' performance is seen to remain stable when building relatively coarse approximations of a given operator preconditioner. Clearly, this has no impact over the solution accuracy as this is only induced by the numerical approximation of the original problem, estimated by Strang's lemma \cite{strang1972variational} and its variants \cite{ern2004theory,di2018third}. Accordingly, we recently proposed the idea of systematically ``\emph{combining distinct precision orders of magnitude inside the resolution scheme}" \cite{escapiljerez} with successful numerical results for boundary element methods in electromagnetics \cite{escapiljerez,kleanthous2020accelerated} and acoustics \cite{fierro2020fast}, despite  hitherto the lack of rigorous proof.

Thus, we aim to provide theoretical grounds for the above observations by considering parameter-dependent perturbed problems and introducing the \emph{bi-parametric} OP paradigm (\Cref{thm:strang}), with bounds on spectral and Euclidean condition numbers with respect to perturbations. We further deduce linear (resp.~super-linear) convergence results for GMRES$(m)$ (resp.~GMRES), and present exhaustive new convergence bounds for iterative solvers when working on Hilbert spaces. Due to their generality, our results apply to diverse research areas: equivalent operators theory \cite{faber1990theory,axelsson2009equivalent,kirby2010functional}, opposite-order OP \cite{winther1980some,stevenson2019uniform}, compact equivalent OP \cite{axelssonsuperlinear,axelsson2018superlinear} and (fast) Calder\'on preconditioning \cite{andriulli2008multiplicative,escapiljerez,antoine2020introduction,fierro2020fast,hiptmair2020preconditioning}. Furthermore, these ideas could be also applied on high frequency wave propagation problems \cite{graham2017domain,galkowski2019wavenumber}, Schwarz preconditioning \cite{sarkis2007optimal,feischl2017optimal} and second-kind Fredholm operator equations \cite{atkinson1976survey,colton2012inverse}.

This manuscript is structured as follows. In \Cref{sec:problem}, we present the abstract PG setting. In \Cref{sec:first_strang} we introduce perturbed forms and state the first Strang's lemma for completeness. Next, we arrive at the bi-parametric OP framework and state our main result in \Cref{sec:bi_parametric}. Finally, we investigate the performance of iterative solvers in \Cref{sec:iterative}, and discuss new research avenues in \Cref{sec:concl}. \Cref{fig:Problem} summarizes constants and problems defined throughout this work.

\begin{figure}
\begin{minipage}{0.4\linewidth}
\begin{table}[H]
\renewcommand\arraystretch{2}
\begin{center}
\footnotesize
\begin{tabular}{
    |>{\centering\arraybackslash}m{1.2cm}
    |>{\centering\arraybackslash}m{2.4cm}
    |>{\centering\arraybackslash}m{1cm}|
    >{\centering\arraybackslash}m{3cm}
    |>{\centering\arraybackslash}m{2cm}|
    }
    \hline
\vspace{0.1cm}
Notation & Value  & Eq. \\ \hline\hline
$\K_\fA$ & $\frac{\|\fa\|}{\gamma_\fA}$  & \eqref{eq:discrete_condition_number}\\\hline
$\KL$ & $\frac{\|\Lambda\|}{\gamma_\Lambda}$  & \eqref{eq:KL}\\ \hline
$\K_\star$ & $\frac{\|\fm\|\|\fn\|\|\fc\|\|\fa\|}{\gamma_\fM\gamma_\fN\gamma_\fC \gamma_\fA}$  & \eqref{eq:kS_star}\\ \hline
$\K_{\star,\mu,\nu}$ & $\K_\star \left(\frac{1 + \mu }{1- \mu}\right)\left(\frac{1 +\nu}{1 - \nu}\right)$ \ &  \eqref{eq:k_mu_nu}\\ \hline
$ \overline{\sigma}_k(\fK)$ & $\frac{1}{k} \sum_{j=1}^k \sigma_j(\fK)$ & \rev{\eqref{eq:partial_singular_values}}\\ \hline 
\end{tabular}
\end{center} 
\end{table} 
\end{minipage}
\begin{minipage}{0.001\linewidth}
\hspace{\fill}
\end{minipage}
\begin{minipage}{0.59\linewidth}
\begin{table}[H]
\renewcommand\arraystretch{1.7}
\begin{center}
\footnotesize
\begin{tabular}{
    |>{\centering\arraybackslash}m{1.2cm}
    |>{\centering\arraybackslash}m{5cm}
    |>{\centering\arraybackslash}m{1cm}|
    }
    \hline
\vspace{0.1cm}
Problem & Problem in matrix form & Eq. \\ \hline\hline
\A &  $\bA \bu = \bb$ &\eqref{eq:matrix_weak}\\\hline
\A$_\nu$ &  $\bA_\nu \bu_\nu = \bb_\nu$ &\eqref{eq:matrix_weak_perturbed}\\\hline
\CA &  $\bM^{-1} \bC \bN^{-1} \bA \bu = \bM^{-1} \bC \bN^{-1}\bb$ &\eqref{eq:prec_matrix_strong}\\\hline
\CA$_{\mu,\nu}$ &   \multirow{2}{*}{$\bM^{-1} \bC_\mu \bN^{-1} \bA_\nu \bu_\nu = \bM^{-1} \bC_\mu \bN^{-1}\bb_\nu$ }&\eqref{eq:prec_matrix_strong_perturbed}\\\cline{1-1}\cline{3-3}
\CA$_{\mu,\nu}^p$& & \eqref{eq:CAmnp}\\ \hline 
\A$^p$& $\bN^{-1} \bA \bu = \bN^{-1} \bb$ &\eqref{eq:Ap}\\ \hline 
\end{tabular}
\end{center} 
\end{table} 
\end{minipage}
\caption{Comprehensive review of the constants (left) and problems (right) defined throughout this manuscript, along with their corresponding introductions. Convergence radius $\Theta_k^{(m)}$ and $\widetilde{\Theta}_k^{(m)}$ for the preconditioned GMRES\rev{($m$)} are defined in \eqref{eq:ConvRates}.}
\label{fig:Problem}
\end{figure}
\section{Continuous, Discrete and Matrix Problem Statements}\label{sec:problem}
Let $X$ and $Y$ be two reflexive Banach spaces and let $\fa \in \mL(X\times Y;\IC)$ be a continuous complex sesqui-linear---weak---form with norm $\|\fa\|$. We tag dual spaces by prime ($'$) and adjoint operators by asterisk (${}^*$). For a linear form $b \in Y'$, the \emph{weak continuous problem} is
\be\label{eq:continuous_weak} 
\text{seek}\quad u \in X \quad\text{such that}\quad \fa(u,v)=b(v),\quad \forall \ v\in Y. 
\ee
Throughout, we assume for each $b\in Y'$ the existence of a unique continuous solution $u$ to \eqref{eq:continuous_weak}. The form $\fa$ induces a---strong---bounded linear operator $\fA\in \mL(X;Y')$ defined through the \emph{dual pairing} in $Y$ as follows \cite[Proposition A.21]{ern2004theory}
\be\label{eq:dual_pairing} 
\langle \fA u, v \rangle_{Y'\times Y}:= a(u,v) ,\quad \forall \ u \in X, ~ \forall \ v \in Y.
\ee
Hence, \eqref{eq:continuous_weak} is equivalent to the \emph{strong continuous problem}: 
\be \label{eq:strongcontprob}
\text{seek}\quad u \in X \quad\text{such that}\quad \fA u = b.
\ee 

Given an index $h>0$, we introduce finite-dimensional \emph{conforming spaces}, i.e.~$X_h \subset X$ and $Y_h\subset Y$, and assume that $\dim(X_h)=\dim(Y_h)=:N$, with $N\rightarrow\infty$ as $h\rightarrow0$. Customarily, $h$ relates to the mesh-size of finite or boundary elements approximations.\footnote{For the sake of simplicity, the problems under consideration are defined for a given $h>0$ although asymptotic considerations are key in proving properties such as \emph{$h$-independent} condition numbers, i.e.~remaining bounded as $h\to 0$ (cf.~\Cref{remark:Asymptotics}).}

The counterpart of \eqref{eq:continuous_weak} is the \emph{weak discrete problem}:
\be\label{eq:discrete_weak}
\text{find}\quad u_h \in X_h \quad\text{such that}\quad \fa(u_h,v_h)=b(v_h),\quad \forall \ v_h\in Y_h,
\ee
The above admits a unique solution $u_h$ \cite[Theorem 2.22]{ern2004theory} if $\fa$ satisfies the discrete inf-sup---Banach-Ne\v{c}as-Babu\v{s}ka (BNB)---condition, for a constant $\gamma_\fA > 0$:
\be\label{eq:BNBh}
\sup_{v_h\in Y_h \setminus \{\boldsymbol{0}\}}\frac{|\fa(u_h , v_h)|}{\n{v_h}_Y} \geq \gamma_\fA \n{u_h}_X > 0 ,~\forall \ u_h \in X_h.
\ee 

\begin{assumption}\label{ass:BNB}
Throughout, we assume that $\fa$ is continuous and satisfies the BNB condition \eqref{eq:BNBh}.
\end{assumption}
Equivalently, we define the discrete operator $\fA_h : X_h \to Y_h'$: 
\be\label{eq:discrete_operator_def}
\langle \fA_h u_h,v_h\rangle_{Y_h'\times Y_h}:= \fa(u_h,v_h),\quad \forall \ u_h \in X_h, \ \forall v_h \in Y_h,
\ee
and $b_h \in Y_h'$ such that $b_h(v_h) := b(v_h)$ for all $v_h \in Y_h$, wherein the norms of $b_h$ and $\fA_h$ are given by (refer to \cite[Section 4.2.3]{sauter2010boundary}):
\be\label{eq:discrete_op_norm}
\|b_h\|_{Y_h'} := \sup_{v_h \in Y_h \setminus \{\boldsymbol{0}\}} \frac{|\fa (u_h,v_h)|}{\n{v_h}_{Y_h}}\quad \text{and} \quad \n{\fA_h}_{X_h \to Y'_h}  : = \sup_{u_h \in X_h \setminus \{\boldsymbol{0}\}} \frac{\n{\fA_h u_h}_{Y_h'}}{\n{u_h}_{X_h}}.
\ee
Consequently, the \emph{strong discrete problem} related to \eqref{eq:strongcontprob} reads
\be\label{eq:discrete_strong}
\text{seek}\quad u_{h} \in X_h \quad\text{such that}\quad \fA_h u_h=b_h.
\ee
One can introduce the \emph{discrete condition number}:
\be\label{eq:discrete_condition_number} 
\kappa (\fA_h) : = \n{\fA_h}_{{X_h\to Y_h'}}\n{\fA^{-1}_h}_{{Y_h'\to X_h}} \leq \gamma_\fA^{-1}\|\fa\|=:\K_\fA,
\ee
with $\K_\fA$ being referred to as \emph{BNB condition number}, not to be confused with the BNB condition \eqref{eq:BNBh}.

Pick bases such that $\textrm{span}\{\varphi_i\}_{i=1}^N = X_h\subset X$ and $\textrm{span}\{\phi_i\}_{i=1}^N= Y_h\subset Y$, and write the corresponding coefficient vectors in $\IC^N$ for the basis expansion in bold letters, e.g.,
\begin{align*}
u_h \in X_h : &\quad u_h = \sum_{i=1}^N u_i \varphi_i , ~\quad \bu := (u_i)_{i=1}^N \in \IC^N,\\
v_h \in Y_h : &\quad v_h = \sum_{i=1}^N v_i \phi_i , ~\quad \bv := (v_i)_{i=1}^N \in \IC^N,
\end{align*}
and build the (stiffness) Galerkin matrix and right-hand side 
$$
\bA :=(\fa(\varphi_j, \phi_i))_{i,j=1}^N,\quad \bb : = (b_h(\phi_i))_{i=1}^N.
$$
It holds that
$$
\langle \fA u_h, v_h\rangle_{Y' \times Y}=\langle \fA_h u_h,v_h\rangle_{Y_{h}'\times Y_{h}} = (\bA \bu, \bv)_2,
$$
where $(\bu,\bv )_2$ denotes the Euclidean inner product in $\IC^N$ with induced norm $\n{\bu}_2=\sqrt{(\bu,\bu)_2}$. The matrix norm is 
$$\n{\bA}_2 : = \max_{\bu \in \IC^N \setminus \{\boldsymbol{0}\}} \frac{ \n{\bA\bu}_2 }{ \n{\bu}_2}.$$

We set $\bA^H:=\overline{\bA}^T$ the conjugate transpose of $\bA$ and define vector and matrix norms induced by the Banach space setting as $\n{\bu}_{X_h}:= \n{u_h}_{X_h}$ and $\n{\bA}_{X_h \to Y_h'}:= \n{\fA_h}_{X_h\to Y_h'}$, for $\fA_h$ in \eqref{eq:discrete_op_norm}. Notice that inclusion $X_h \subset X$ ensures that $\n{\bu}_{X_h} = \n{\bu}_X =: \n{u_h}_X $.

Consequently, \eqref{eq:discrete_weak} and \eqref{eq:discrete_strong} correspond to the \emph{matrix problem} referred to\footnote{In the following, notation \textup{(\!($\cdot$)\!)} denotes matrix equations.} as \A:
\be\label{eq:matrix_weak}
\text{\A}: \quad\text{Seek}\quad \bu \in \IC^N  \quad \text{such that} \quad\bA\bu = \bb.
\ee
Next, we introduce $\Lambda_h$ the \emph{synthesis operator} for $X_h$:
\be\label{eq:synthesis}
\begin{split}
\Lambda_h : \IC^N &\to X_h\\
\bu& \mapsto  u_h,
\end{split}
\ee
along with strictly positive constants for $h>0$
\be 
\gamma_{\Lambda_{h}}  : =  \inf_{u_h\in X_h \setminus \{\boldsymbol{0}\}}\frac{\n{u_h}_X}{\n{\bu}_2} \quad\textup{and}\quad \|\Lambda_{h} \|: = \sup_{u_h\in X_h \setminus \{\boldsymbol{0}\}} \frac{\n{u_h}_X}{\n{\bu}_2}.
\ee
Notice that, for any $u_h\in X_h$, it holds that \cite[Section 2.3]{ern2006evaluation}
$$
\gamma_{\Lambda_{h}} \n{\bu}_2 \leq \n{u_h}_X \leq \|\Lambda_{h} \| \n{\bu}_2,
$$
and set
\be \label{eq:KL}
\K_{\Lambda_h} : = \frac{\|\Lambda_h\|}{\gamma_{\Lambda_h}}.
\ee

\begin{remark}\label{remark:synthesis}
One should observe the explicit use of $h$-subscripts for the synthesis operator. Indeed, while discrete inf-sup conditions are generally bounded as $h$ tends to zero, the bounds $\|\Lambda_h\|$ and $\gamma_{\Lambda_h}$ are not. For example, let $\Omega \subset\IR^d$, $d=2,3$ be a smooth bounded Lipschitz domain \cite[Section 2]{steinbach2007numerical} with boundary $\Gamma :=\partial \Omega$. For $D$ being either $\Gamma$ or $\Omega$ and $s\in [0,1]$, we introduce the Sobolev space $H^s(D)$ \cite{steinbach2007numerical} and let $X:=H^s(D)$. We assume that $D$ is decomposed into a shape regular, locally quasi-uniform mesh $\mathcal{T}$ \cite[Section 9.1]{steinbach2007numerical} with elements $\tau \in \mathcal{T}$. Set $h_\tau$ as the diameter of each element $\tau \in \mathcal{T}$, along with $h_\textup{min}:= \min_{\tau \in \mathcal{T}}h_\tau $ and $h \equiv h_\textup{max}:= \max_{\tau \in \mathcal{T}} h_\tau $, and introduce a nodal $\mathcal{C}^0$-Lagrangian basis  \cite{ainsworth1999conditioning} on $\mathcal{T}$ as $\textup{span}\{\phi_i\}_{i=1}^N= X_h \subset X$, for any $N(h) \in \mathbb{N}$. For all $u_h \in X_h$, there holds that \cite[Sections 4.4 and 4.5]{sauter2010boundary}:
\be 
 \quad Ch_\textup{min}^\frac{d}{2} \n{\bu}_2 \leq C \n{u_h}_{L^2(D)} \leq  \n{u_h}_{H^s(D)} \leq Ch_\textup{min}^{-s} \n{u_h}_{L^2(D)} \leq C h_\textup{min}^{-s} h_\textup{max}^{\frac{d}{2}} \n{\bu}_2.
\ee
Consequently, one obtains
\be\label{eq:hdep1}
\gamma_{\Lambda_h} \geq C  h_\textup{min}^\frac{d}{2} ,\quad \|\Lambda_h\|\leq C h_\textup{min}^{-s}h_\textup{max}^{\frac{d}{2}}\quad \text{and} \quad \KL \leq C \left(\frac{h_\textup{max}}{h_\textup{min}}\right)^{\frac{d}{2}}h_\textup{min}^{-s}.
\ee
For $D=\Gamma$ and $H^s(\Gamma)$, with $s\in [-1,1]$, one has
\be \label{eq:hdep2}
\KL \leq C \left(\frac{h_\textup{max}}{h_\textup{min}}\right)^{\frac{d}{2}}h_\textup{min}^{-|s|}.
\ee
In this case, one can see the synthesis operator's explicit $h$-dependence via \eqref{eq:hdep1} and \eqref{eq:hdep2}. A similar situation holds in the case of N\'ed\'elec and Raviart-Thomas (Rao-Wilton-Glisson) elements applied in electromagnetic scattering (cf.~\cite{HJM15} and references therein).
\end{remark}
For the remainder of this work, we will make extensive use of the spectral and Euclidean condition numbers, $\kappa_S (\bA)$ and $\kappa_2 (\bA)$, respectively,  defined as
\be 
\begin{split}\label{eq:kappas}
\kappa_S (\bA) := \varrho(\bA) \varrho (\bA^{-1}) = \frac{|\lambda_\textup{max}(\bA)|}{|\lambda_\textup{min}(\bA)|} \quad \text{and}\quad
\kappa_2 (\bA) := \n{\bA}_2 \n{\bA^{-1}}_2,
\end{split}
\ee
with $\varrho(\bA):= |\lambda_\textup{max}(\bA)|$ being the spectral radius of $\bA$. We denote the spectrum of $\bA$ by $\mathfrak{S}(\bA)$. Since the spectral radius is bounded by any norm on $\IC^N$, we set, for any $\fQ_h:X_h\to X_h$ with matrix representation $\bQ$, the Banach space induced norm $\n{\bQ}_X := \n{\bQ}_{X\to X}$, with $\n{\bQ}_{X} = \n{\bQ}_{X_h} =: \n{\bQ}_{X_h\to X_h}$ by inclusion $X_h \subset X$, leading to $\varrho(\bQ) \leq \n{\bQ}_{X}$. The latter is key in proving the operator preconditioning result in \Cref{thm:opprecond}.

As mentioned in \Cref{sec:Introduction}, we are concerned with the consequences of perturbing the above sesqui-linear and linear forms over discretization spaces as it occurs when employing finite-arithmetic, numerical integration or compression algorithms. To this end, we give a notion of admissible perturbations needed for the ensuing analysis.
\begin{definition}[($h,\nu$)-perturbation]\label{def:perturbation}Let $\nu \in [0,1)$ and $h>0$ be given. We say that $\fa_\nu \in \mL(X \times Y;\IC)$ is a $(h,\nu)$-perturbation of $\fa$ if it belongs to the set $\Phi_{h,\nu}(\fa)$:
\be
 \fa_\nu \in \Phi_{h,\nu}(\fa)  \iff \gamma_\fA^{-1} \left| \fa (u_h,v_h) - \fa_\nu(u_h,v_h)\right| \leq  \nu \n{u_h}_X \n{v_h}_Y ,~ \forall \ u_h \in X_h, \ \forall \ v_h\in Y_h.
\ee
Similarly, $b_{\nu}\in Y'$ is called a $(h,\nu)$-perturbation of the linear form $b$ if it belongs to the set $\Upsilon_{h,\nu}({b})$ defined as
$$ b_{\nu} \in \Upsilon_{h,\nu}(b)  \iff \left|b(v_h) - b_{\nu}(v_h) \right|\leq \nu \n{b_h}_{Y_h'}\n{v_h}_{Y_h},\quad\forall \ v_h\in Y_h. $$
We identify $\fa_0$ and $b_0$ with $\fa$ and $b$, respectively. 
\end{definition}
The $(h,\nu)$-perturbation formalism allows to control precisely the perturbed sesqui-linear (resp.~linear) form.

\begin{proposition}\label{prop:perturbed}
Consider $\fa_\nu \in \Phi_{h,\nu} (\fa)$. Then, $\fa_\nu$ has a discrete inf-sup condition and is continuous, with corresponding constants $\gamma_{\fA_\nu }$, $\|\fa_\nu\|$, satisfying
\be 
\gamma_{\fA_\nu } \geq \gamma_\fA( 1- \nu) \quad \text{and} \quad \|\fa_\nu\|\leq \|\fa\| + \nu\gamma_\fA \leq \|\fa\|(1+\nu).
\ee
\end{proposition} 
\begin{proof} 
For any $u_h \in X_h$, it holds that
\be\label{eq:inf_sup1}
\begin{split}
\sup_{v_h \in Y_h \setminus \{\boldsymbol{0}\}} \frac{|\fa_\nu (u_h,v_h)|}{\n{v_h}_Y} &\geq\sup_{v_h \in Y_h \setminus \{\boldsymbol{0}\}}\left( \frac{|\fa (u_h,v_h)|}{\n{v_h}_Y} - \frac{|\fa (u_h,v_h)-\fa_\nu(u_h,v_h)|}{\n{v_h}_Y} \right)\\
& \geq \gamma_\fA \n{u_h}_X - \gamma_\fA \nu \n{u_h}_X = \gamma_\fA (1-\nu)\n{u_h}_X
\end{split}
\ee
by \Cref{ass:BNB} and \Cref{def:perturbation}. Similarly, for any $u_h\in X_h$ and $v_h \in Y_h$, one has
\begin{align*}
|\fa_\nu (u_h,v_h)|&\leq |\fa (u_h,v_h)| + |\fa_\nu (u_h,v_h)-\fa(u_h,v_h)| \\
&\leq ( \|\fa\| +\nu  \gamma_\fA )\n{u_h}_X \n{v_h}_Y  \leq  \|\fa\|(1 +\nu  )\n{u_h}_X \n{v_h}_Y,
\end{align*}
as stated.\end{proof}
\begin{remark}
Though the sets of admissible perturbations $\Phi_{h,\nu}(\fa)$ and $\Upsilon_{h,\nu}(b)$ depend on $h$, the perturbed forms remain continuous.  Also, for a given $h$, one may choose different parameters for each set.
\end{remark}

Set $\nu \in [0,1)$ and introduce perturbations $\fa_\nu\in\Phi_{h,\nu}(\fa)$ and $b_\nu\in\Upsilon_{h,\nu}(b)$. We arrive at the \emph{perturbed weak discrete problem}:
\be\label{eq:perturbed_discrete}
\text{seek}\quad u_{h,\nu} \in X_h \quad\text{such that}\quad \fa_\nu(u_{h,\nu},v_h)=b_\nu(v_h),\quad \forall \ v_h\in Y_h,
\ee
with strong discrete counterpart
\be\label{prob:unprecond}
 \text{find}\quad u_{h,\nu} \in X_h \quad\text{such that}\quad \fA_{h,\nu} u_{h,\nu} =b_{h,\nu},
\ee
and matrix form
\be\label{eq:matrix_weak_perturbed}
\text{\A}_\nu: \quad\text{Seek}\quad \bu_\nu \in \IC^N  \quad \text{such that} \quad\bA_\nu\bu = \bb_\nu.
\ee
Notice that \A$_0=\A$. Moreover, one can combine \Cref{prop:perturbed} with \cite[Theorem 2.22]{ern2004theory} to obtain the next result.
\begin{proposition}\label{prop:well_posedness}For $\nu \in [0,1)$, \A$_\nu$ admits a unique solution.
\end{proposition}
\section{First Strang's lemma for perturbed forms}\label{sec:first_strang}
We start by characterizing the error between continuous and discrete solutions for the unperturbed version of \A~recalling C\'ea's lemma \cite{cea1964approximation, ern2004theory}.
\begin{lemma}[C\'ea's Lemma {\cite[Lemma 2.28]{ern2004theory}}]\label{lemmaCea}
Let $u\in X$ and $u_h\in X_h$ be the solutions to \eqref{eq:continuous_weak} and \eqref{eq:discrete_weak}, respectively. Then, one has
\be
\n{u - u_h}_X \leq  \left(1 + {\K_\fA} \right)\inf_{w_h \in X_h} \n{u-w_h}_X,
\ee
with $\K_\fA$ defined in \eqref{eq:discrete_condition_number}.
\end{lemma}
\begin{remark}
This fundamental result highlights the importance of the BNB condition number. It shows that if the problem has poor intrinsic conditioning for either continuous or discrete settings, then the quasi-optimality constant $(1 + \K_\fA)$ will be large and the solution $u_h$ far from the best approximation error. Observe that in \Cref{lemmaCea} both sesqui-linear and linear forms are computed exactly.
\end{remark}
Next, we present a modified version of the above lemma for perturbed problems \A$_\nu$. 
\begin{lemma}[First Strang's Lemma]\label{lemma:strang}
Set $\nu\in[0,1)$ and let $u_{h,\nu}\in X_h$ and $u\in X$ be the unique solutions to \eqref{prob:unprecond} and \eqref{eq:continuous_weak}, respectively. It holds that
\begin{align*}
\n{u - u_{h,\nu}}_X & \leq \inf_{w_h\in X_h} \left(\left(1 + \frac{{\K_\fA}}{1-\nu}\right) \n{u - w_h}_X +\frac{\nu}{1-\nu} \n{w_h}_X\right) + \frac{\nu}{\gamma_{\fA}(1-\nu)}\n{b_h}_{Y_h'}\\
&\leq \left( 1 + \K_\fA\right) \left(1 + \frac{{\K_\fA}}{1-\nu}\right)\inf_{w_h\in X_h} \n{u - w_h}_X + \frac{2\nu}{\gamma_{\fA}(1-\nu)}\n{b_h}_{Y_h'}.
\end{align*}
\end{lemma} 
\begin{proof}
For any $w_h\in X_h$ and for all $v_h \in Y_h$, it holds that
\begin{align*} 
\fa_\nu (u_{h,\nu} - w_h, v_h ) &= b_\nu(v_h) -\fa_\nu (w_h,v_h)  +\fa(w_h,v_h) + \fa(u-w_h ,v_h) - b(v_h)\\
& =  \fa(u-w_h ,v_h) + (\fa(w_h,v_h) -\fa_\nu (w_h,v_h) ) + (b_\nu(v_h) -b(v_h)),
\end{align*}
leading to
\be\label{eq:strang1}
\gamma_{\fA_\nu}\n{u_{h,\nu} - w_h}_X \leq \|\fa\| \n{u- w_h}_X + \nu \gamma_\fA \n{w_h}_X+ \n{b_h-b_{h,\nu}}_{Y'_h}
\ee
by \Cref{ass:BNB} and \Cref{def:perturbation}. Next, by combining the triangle inequality, and \eqref{eq:strang1}, one derives
\begin{align*}
\n{u - u_{h,\nu}}_X &\leq \n{u - w_h}_X + \n{w_h - u_{h,\nu}}_X \leq \left(1 +  \frac{\|{\fa}\| }{\gamma_{\fA_\nu}}\right)\n{u - w_h}_X + \nu \frac{\gamma_\fA}{\gamma_{\fA_\nu}}\n{w_h}_X +\frac{1}{\gamma_{\fA_\nu}} \n{b_h-b_{h,\nu}}_{Y_h'},
\end{align*}
and, since $w_h$ is arbitrary in $X_h$, there holds
\begin{align*}
\n{u - u_{h,\nu}}_X &\leq \frac{1}{\gamma_{\fA_\nu}}\n{b_h-b_{h,\nu}}_{Y'_h}+ \inf_{w_h \in X_h}\left( \left(1 + \frac{\|{\fa}\|}{\gamma_{\fA_\nu}}\right) \n{u-w_h}_X + \frac{\gamma_{\fA}}{\gamma_{\fA_\nu}}\nu\n{w_h}_{X}\right)\\
&\leq  \frac{\nu}{\gamma_\fA(1-\nu)}\n{b_h}_{Y_h'} +\inf_{w_h \in X_h}\left( \left(1 + \frac{{\K_\fA}}{1-\nu}
\right) \n{u-w_h}_X + \frac{\nu}{1-\nu} \n{w_h}_{X}\right)\\
&\leq  \frac{\nu}{\gamma_\fA(1-\nu)}\n{b_h}_{Y_h'} +\left(1 + \frac{{\K_\fA}}{1-\nu}
\right) \n{u-u_h}_X   + \frac{\nu}{1-\nu} \n{u_h}_{X}\\
&\leq  \frac{2\nu}{\gamma_\fA(1-\nu)}\n{b_h}_{Y_h'} +\left(1+ \K_\fA\right)\left( 1 + \frac{{\K_\fA}}{1-\nu}
\right) \inf_{w_h \in X_h}\n{u-w_h}_X ,
\end{align*}
as stated, by recalling the continuous dependence on $b$ for $u_h$ solution of \eqref{eq:discrete_weak}, i.e.~$\n{u_h}_X \leq \frac{1}{\gamma_\fA}\n{b_h}_{Y_h'}$, and by application of \Cref{lemmaCea}.
\end{proof}
\begin{remark}\label{remark:stabilityAnu}
Since
\begin{align*}
\kappa(\fA_{h,\nu})\leq \frac{\|\fa_{\nu}\|}{\gamma_{\fA_\nu}}=: \K_{\fA_\nu} \leq \fK_\fA \frac{1 + \nu}{1-\nu},
\end{align*}
one can expect the discrete (resp.~BNB) condition number of $\fa_\nu$ to be stable with respect to small perturbations, as 
$\K_{\fA_\nu}= \K_\fA (1 - 2\nu + o(\nu))$ for $\nu\ll 1$. For $\nu \ll 1$, \Cref{lemma:strang} shows that the perturbation implies: a best approximation error term with quasi-optimality constant $(1 + \K_\fA)^{2}$, and $\mO(\nu)$ errors induced by the perturbed sesqui-linear form and right-hand side (cf.~\cite[Sections 2 and 3]{escapiljerez}).
\end{remark}
\begin{remark}Observe that contrary to C\'ea's lemma, \Cref{lemma:strang} does not invoke the solution to the continuous perturbed problem:
\be\label{eq:continuous_weak_perturbed} 
\text{seek}\quad u_\nu \in X \quad\text{such that}\quad \fa_\nu(u_\nu,v)=b_\nu(v),\quad \forall \ v\in Y. 
\ee
Assuming the existence of a unique continuous solution $u_\nu$ to \eqref{eq:continuous_weak_perturbed}, \Cref{lemmaCea} and \Cref{remark:stabilityAnu} lead to the following quasi-optimal bound:
\be 
\n{u_\nu - u_{h,\nu}}_X \leq  \left(1 + \K_\fA\frac{1+\nu}{1-\nu} \right)\inf_{w_h \in X_h} \n{u_\nu-w_h}_X.
\ee
\end{remark}
\section{Bi-parametric Operator Preconditioning}\label{sec:bi_parametric}
We complete the setting in \Cref{sec:problem} by introducing preconditioners. To this end, let $V$ and $W$ be two reflexive Banach spaces. We consider an operator $\fc \in \mL (V \times W ; \IC)$ as well as pairings $\fn\in \mL (V\times Y; \IC)$ and $\fm \in \mL ( X\times W; \IC)$. These forms induce operators $\fC: V \to W'$, $\fN : V \to Y'$ and $\fM: X\to W'$. With these, we state the preconditioned version of the operator equation \eqref{eq:strongcontprob}:
\be
\text{seek}\quad u\in X \quad\text{such that}\quad \fP \fA u = \fP b,\quad \text{with}  \quad \fP : = \fM^{-1}\fC\fN^{-1}.
\ee

We refer the readers to \Cref{fig:Problem} and to the previous diagram in \eqref{eq:OP_PG} for an overview of domain mappings and functional spaces for OP-PG.

For our new spaces, we set conforming finite-dimensional spaces $V_h\subset V$ and $W_h\subset W$ of the same dimension $N$ as for $X_h$ and $Y_h$.

\begin{assumption}
We assume that \rev{$\fc:X \times Y$, $\fn: V \times Y$ and $\fm:X \times W$} satisfy a discrete inf-sup condition (cf.~\eqref{eq:BNBh}) over the approximation spaces, with constants $\gamma_\fC, \gamma_\fN$ and $\gamma_\fM$, respectively.
\end{assumption}

Consequently, the \emph{strong discrete preconditioned problem} 
\be
\label{eq:prec_discrete_strong}\text{seek}\quad u_{h} \in X_h \quad\text{such that}\quad \fP_h\fA_h u_h= \fP_h b_h,\quad \text{with} \quad \fP_h:=\fM_h^{-1} \fC_h \fN_h^{-1},\ee
is well posed, by the same arguments as in \Cref{prop:well_posedness}. As in \Cref{sec:problem}, we now pick bases $\{\psi_i\}_{i=1}^N \subset V_h$ and $\{\xi_i\}_{i=1}^N \subset W_h$ of $V_h$ and $W_h$, and build the Galerkin matrices 
\be 
\bC : = {((\fc( \psi_j ,\xi_i ))}_{i,j = 1}^N, \quad \bM : = {((\fm( \varphi_j , \xi_i ))}_{i,j = 1}^N \quad \text{and} \quad \bN : = {((\fn( \psi_j , \phi_i ))}_{i,j = 1}^N.
\ee
Therefore, we arrive at the matrix problem:
\be\label{eq:prec_matrix_strong}
\text{\CA}: \quad\text{find}\quad \bu \in \IC^N  \quad \text{such that}\quad \bP  \bA \bu = \bP \bb,\quad\text{with} \quad \bP:= \bM^{-1} \bC \bN^{-1}.
\ee 
\begin{table}[t]
\renewcommand\arraystretch{1.7}
\begin{center}
\footnotesize
\begin{tabular}{
    |>{\centering\arraybackslash}m{2cm}
    |>{\centering\arraybackslash}m{2cm}
    |>{\centering\arraybackslash}m{2cm}
    |>{\centering\arraybackslash}m{3cm}
    |>{\centering\arraybackslash}m{2cm}|
    }
    \hline
\vspace{0.1cm}
&Operator & sesqui-linear form & Matrix  & Constants\\ \hline\hline
Impedance & $\fA: X \to Y'$ & $\fa : X \times Y$ &$\bA : [Y_h\times X_h]$  & $\gamma_\fA, \|\fa\|$\\\hline
Preconditioner& $\fC: V\to W'$ & $\fc :V \times W$ &$\bC : [W_h\times V_h]$ & $\gamma_\fC, \|\fc\|$\\\hline\hline
Pairing $\fA$ & $\fN: V \to Y'$ & $\fn : V \times Y $ &$\bN^{-1} : [V_h\times Y_h]$& $\gamma_\fN, \|\fn\|$ \\\hline
Pairing $\fC$ & $\fM: X \to W'$ & $\fm : X\times W $ &$\bM^{-1} : [X_h\times W_h]$& $\gamma_\fM, \|\fm\|$ \\\hline
 \end{tabular}
\caption{Overview of functional spaces for OP-PG. We specify spaces for the corresponding continuous operators and sesqui-linear forms, along with their induced discrete matrices, continuity and discrete-inf sup constants. Brackets for matrices indicate the spaces associated to rows $\times$ columns.}
\label{tab:Problem}
\end{center} 
\end{table}
\begin{remark}\label{remark:OPalgebra}
As hinted in \cite{betcke2020product}, OP allows to obtain an equivalent representation for both the discrete and matrix settings, referred to as Galerkin product algebra. Indeed, introduce a unique $v_h \in V_h$ such that $ \fN_h v_h = b_h$, $w_h:=\fC_h v_h \in W'_h$, and a unique $q_h \in X_h$ such that $\fM_h q_h =  w_h$. We obtain that 
\begin{align*}
\fA_h u_h = b_h = \fN_h v_h \quad \Rightarrow \quad\fN_h^{-1} \fA_h u_h = v_h \in V_h,\\
\fC_h v_h =  w_h = \fM_h q_h \quad \Rightarrow\quad \fM_h^{-1} \fC_h v_h = q_h \in X_h,
\end{align*}
leading to matrix counterparts
\begin{align*}
\bA \bu &= \bN \bv\quad \Rightarrow \quad \bN^{-1} \bA \bu = \bv, \\
\bC \bv & = \bM \bq \quad \Rightarrow \quad \bM^{-1} \bC \bv = \bq.
\end{align*}
Hence
\be \label{eq:OPalgebra}
q_h = \fP_h \fA_h u_h\quad \text{with basis expansion} \quad \bq = \bP \bA \bu.
\ee
Consequently, $u_h = (\fP_h \fA_h)^{-1}q_h$ is with basis expansion $\bu = (\bP \bA)^{-1}\bq$.
\end{remark}
We state the following estimates for the condition numbers of $\bP \bA$.
\begin{theorem}[Estimates for OP-PG]
\label{thm:opprecond}For problem \CA~given in \eqref{eq:prec_matrix_strong}, the spectral condition number is bounded as  
\be\label{eq:kS_star}
\kappa_S (\bP\bA) \leq \kappa (\fP_h \fA_h) \leq  \frac{\|\fm\|\|\fn\|\|\fc\|\|\fa\|}{\gamma_\fM\gamma_\fN\gamma_\fC \gamma_\fA} =:\K_\star.
\ee
Furthermore, the Euclidean condition number satisfies
\be\label{eq:k2_star}
\kappa_2 (\bP\bA) \leq \K_\star \left(\frac{\n{\Lambda_h}}{\gamma_{\Lambda_h}}\right)^2= \K_\star \KL^2,
\ee
with $\KL$ introduced in \eqref{eq:KL}.
\end{theorem}
\begin{proof}
Remark that, for any $u_h \in X_h$, it holds that
\be \label{eq:spectral}
\frac{\gamma_\fC \gamma_\fA}{\|\fm\|\|\fn\|}\|u_h\|_X\leq \n{\fP_h\fA_h u_h }_{X} \leq \frac{\|\fc \| \|\fa\|}{\gamma_\fN \gamma_\fM}\n{u_h}_X.
\ee
Let us introduce $\bu$ linked to $u_h$ so as to deduce that 
\be 
\n{\fP_h\fA_h}_{X} = \n{\bP \bA}_{X}\leq \frac{\|\fc \| \|\fa\|}{\gamma_\fN \gamma_\fM}\quad \text{and}\quad \n{(\fP_h\fA_h)^{-1}}_X = \n{(\bP \bA)^{-1}}_{X}\leq\frac{\|\fm\|\|\fn\|}{\gamma_\fC \gamma_\fA},
\ee 
which leads to the stated result for the spectral condition number given in \eqref{eq:kappas}, since $\varrho(\bP \bA ) \leq \n{\bP \bA}_{X}$ and $\varrho((\bP \bA)^{-1})\leq \n{(\bP \bA)^{-1}}_{X}$.

For the Euclidean condition number, we employ the synthesis operator $\Lambda_h$, introduced in \eqref{eq:synthesis}, and \eqref{eq:spectral}, to derive
\be 
\frac{1}{\|\Lambda_h\|} \left(\frac{\gamma_\fC \gamma_\fA}{\|\fm\|\|\fn\|} \right)\n{u_h}_X \leq \n{\bP \bA\bu}_2 \leq \frac{1}{\gamma_{\Lambda_h}} \left(\frac{ \|\fc\|\|\fa\|}{\gamma_\fM \gamma_\fN} \right) \n{u_h}_X,
\ee
yielding
\be 
\frac{\gamma_{\Lambda_h} }{\|\Lambda_h\|} \left(\frac{\gamma_\fC \gamma_\fA}{\|\fm\|\|\fn\|} \right)\n{\bu}_2 \leq  \n{\bP \bA\bu}_2 \leq \frac{\|\Lambda_h\|}{\gamma_{\Lambda_h}} \left(\frac{ \|\fc\|\|\fa\|}{\gamma_\fM \gamma_\fN} \right)\n{\bu}_2,
\ee
providing the second result.
\end{proof}

As mentioned in \Cref{sec:Introduction}, the abstract formulation in \Cref{thm:opprecond} for OP-PG encompasses the following important cases:
\begin{enumerate}
\item[(i)] OP-BG \cite{hiptmair2006operator,christiansen2000preconditionneurs}: $X=Y$, $V =W$ and $\fN := \fM^*$ (cf.~\eqref{eq:OP_BG});
\item[(ii)] Opposite-order OP \cite{andreev2012stability}: $Y = X'$ and $W = V'$ and $\fM = \fN := \fI$ (cf.~\eqref{eq:OP_OO}).
\end{enumerate}

We are now ready to introduce perturbed sesqui-linear forms and their preconditioners. In the spirit of \eqref{eq:perturbed_discrete}, we consider the family of \emph{bi-parametric} perturbed preconditioned problems. 

For two parameters $\mu,\nu\in [0,1)$, we define $\fc_\mu \in \Phi_{h,\mu}(\fc)$, $\fa_\nu \in \Phi_{h,\nu}(\fa)$, and $b_\nu \in \Upsilon_{h,\nu}(b)$. The \emph{perturbed preconditioned} problem reads 
\be\label{prob:precond}
\text{find}\quad u_{h,\nu} \in X_h \quad\text{such that}\quad \fP_{h,\mu} \fA_{h,\nu} u_{h,\nu} =\fP_{h,\mu} b_{h,\nu},\quad \text{with}\quad \fP_{h,\mu} : = \fM^{-1}_h\fC_{h,\mu},\fN^{-1}_h,
\ee
with corresponding matrix form
\be\label{eq:prec_matrix_strong_perturbed}
\text{\CA}_{\mu,\nu}: \quad\text{seek}\quad \bu_\nu \in \IC^N  \quad \text{such that}\quad \bP_\mu \bA_\nu \bu_\nu =\bP_\mu \bb_\nu,\quad \text{with}\quad \bP_\mu : =\bM^{-1} \bC_\mu \bN^{-1}.
\ee
Naturally, \CA$_{0,0}=$\CA. In practice, one seeks the preconditioner parameter $\mu$ to be much larger than the original system's accuracy $\nu$ while retaining the convergence properties. Indeed, we can now state our main result.

\begin{theorem}[Bi-Parametric Operator Preconditioning]\label{thm:strang}
For the problem \CA$_{\mu,\nu}$, given in \eqref{eq:prec_matrix_strong_perturbed} for $\mu,\nu\in[0,1)$ and $h>0$, the spectral condition number is bounded as
\be\label{eq:kS_mu_nu}\kappa_S (\bP_\mu\bA_\nu) \leq
 \K_\star \left(\frac{1+ \mu}{1-\mu}\right)\left(\frac{1+\nu}{1-\nu}\right)= :\K_{\star,\mu,\nu}
\ee
and the Euclidean condition number satisfies
\be\label{eq:k_mu_nu}\kappa_2 (\bP_\mu \bA_\nu) \leq \K_{\star,\mu,\nu} \K_{\Lambda_h}^2,\ee
with $\K_{\star}$ and $\K_{\Lambda_h}$ defined in \eqref{eq:KL} and \eqref{eq:kS_star}, respectively.
\end{theorem}
\begin{proof}Application of \Cref{prop:perturbed} to $\fc_\mu$ and $\fa_\nu$ leads to:
\be 
\forall \ u_h \in X_h , \quad (1-\mu)(1-\nu)\frac{\gamma_\fC \gamma_\fA}{\|\fm\|\|\fn\|} \n{u_h}_X \leq \n{\fP_{h,\mu} \fA_{h,\nu} u_h}_X\leq \frac{\|\fc\|\|\fa\|}{\gamma_\fA\gamma_\fC}(1+\mu)(1+\nu)\n{u_h}_X,
\ee 
from where one derives the result for the spectral condition number following the proof of \Cref{thm:opprecond}. For the Euclidean condition number, the proof is similar modulo the term $\KL$ due to the synthesis operator.
\end{proof}

\begin{remark}
\Cref{thm:strang} provides bounds for both spectral and Euclidean condition numbers. Notice that \eqref{eq:k_mu_nu} involves the synthesis operators in $X_h$ (see \Cref{remark:synthesis}). Moreover, it holds that $\K_{\star,\mu,\nu} = \K_{\star,\nu,\mu}$, and $\K_{\star,\mu,\nu}$ does not involve cross-terms in $\mu$ and $\nu$. Remark that \eqref{eq:kS_mu_nu} is a sharper estimate than the previous bound in \cite[Proposition 1]{escapiljerez}. Also, we have assumed $\fM$ and $\fN$ to be exact or unperturbed but one could also extend the above results to account for perturbed pairings.
\end{remark} 

\Cref{thm:strang} constitutes the formal proof of the effectiveness of preconditioning with low-accuracy approximations hinted, for instance, by Bebendorf in \cite[Section 3.6]{bebendorf2008hierarchical}. To illustrate this, assume that the best approximation error in \Cref{lemmaCea} converges at a rate $\mO(h^r)$, $r >0$. First, \Cref{thm:strang} shows that one can set $\nu = \mO(h^{r})$ to preserve the convergence rate. Second, one can relax $\mu$ by setting a bounded $\mu=\mO(1)$ guaranteeing a bounded spectral condition number.
Consequently, the result suggests using different parameters for the assembly of $\bP_\mu$ and $\bA_\nu$. For example, one can keep standard Galerkin methods for building stiffness matrices with preconditioners built using coarser Galerkin approximations \cite{escapiljerez,kleanthous2020accelerated,fierro2020fast}, collocation methods \cite{atkinson1976survey}, compression techniques \cite{bebendorf2008hierarchical,bebendorf2009recompression}, or feedforward neural networks \cite{meade1994numerical,sappl2019deep}.

\section{Iterative Solvers Performance: Hilbert space setting}\label{sec:iterative}
Throughout \Cref{sec:iterative}, we restrict ourselves to $X\equiv H$ with $H$ being a Hilbert space with inner product ${(\cdot,\cdot)}_{{H}}$ and ${\|\cdot\|}_{H} = \sqrt{{(\cdot,\cdot)}_{H}}$. We set $\bH := ({(\varphi_j,\varphi_i)}_H)_{i,j=1}^N$, being Hermitian positive definite with $\{\varphi_i\}_{i=1}^N$ defined in \Cref{sec:problem}, satisfying
\be \label{eq:def_inner_product}
\forall \ u_h,v_h \in X_h, \quad {(u_h,v_h)_H} = {\langle \mathsf{R} u_h, v_h\rangle}_{H' \times H} = (\bH \bu,\bv)_2 =  :{(\bu,\bv)}_H,
\ee
where $\mathsf{R}$ is the isometric Riesz-isomorphism ${H}\rightarrow H'$ \cite[Section 3]{hiptmair2006operator}.

We aim at detailing how the context of \Cref{thm:opprecond} and \Cref{thm:strang} transfers onto the behavior of iterative solvers such as GMRES under the above Hilbertian setting. To this end, the following matrix properties will prove useful.

\subsection{Matrix properties: $H$-FoV}

For any $\bQ \in \IC^{N\times N}$, $N\in \mathbb{N}$, we introduce $\mF_H(\bQ)$, the matrix $H$-FoV of $\bQ$---also referred to as $H$-numerical range---and $\mV_H(\bQ)$, the distance of $\mF_H(\bQ)$ from the origin
\be\label{eq:numerical_range} 
\mF_H(\bQ):=\left \{\frac{{(\bQ \bu,\bu)}_H}{(\bu,\bu)_H} :\bu \in \IC^N \setminus \{\boldsymbol{0}\}\right \} \quad \text{and} \quad \mV_H(\bQ) := \min_{z \in \mF_H(\bQ) } |z| = \min_{\bu \in \IC^N \setminus \{\boldsymbol{0}\}} \frac{|(\bQ \bu,\bu)_H|}{(\bu,\bu)_H}.
\ee
Likewise, we introduce $\mF_2(\bQ)$, or equivalently 2-FoV, and $\mV_2(\bQ)$. Moreover, for any $\fQ_h : X_h \to X_h$, we set the discrete $H$-FoV and $\mV_H(\fQ_h)$:
\be\label{eq:discrete_HFOV}
\mF_H(\fQ_h):=\left \{\frac{{(\fQ_h u_h,u_h)}_H}{(u_h, u_h)_H} :u_h\in X_h \setminus \{\boldsymbol{0}\}\right \} \quad \text{and} \quad \mV_H(\fQ_h) := \inf_{u_h \in X_h \setminus \{\boldsymbol{0}\}} \frac{|(\fQ_h u_h,u_h)_H|}{(u_h,u_h)_H} .
\ee
We recall that the $H$-adjoint of $\bQ$ is $\bQ^{\star}:=\bH^{-1}\bQ^H\bH$, and that $\bQ$ is said to be \emph{$H$-normal} if $\bQ$ commutes with $\bQ^{\star}$ \cite[Section 2.2.1.1]{axelsson2009equivalent}.

The matrix $H$-FoV (and $2$-FoV) being key in describing the linear convergence of GMRES$(m)$, we aim at giving a further insight on these sets. Following \cite[Section 4]{benzi2016localization}, we state some useful properties of the matrix $H$-FoV. 
\begin{lemma}[Properties of the matrix $H$-FoV $\mF_H(\bQ)$]\label{lemma:numrange}
Consider any $\bQ,\bH \in \IC^{N\times N}$, $N\in \mathbb{N}$, with $\bH$ being a Hermitian positive definite matrix. The following properties hold:
\begin{enumerate}
  \item[(i)] $\mF_H(\bQ) = \mF_2(\bH^\frac{1}{2}\bQ \bH^{-\frac{1}{2}})$;
  \item[(ii)] Spectral containment: $\mathfrak{S} (\bQ) \subset \mF_H( \bQ)$;
  \item[(iii)] $H$-normal matrices: If $\bQ$ is $H$-normal, then $\mF_H(\bQ)= \textup{Conv}(\mathfrak{S} (\bQ))$ the convex hull of $\mathfrak{S}(\bQ)$;
  \item[(iv)] $\mF_H(\bQ)$ is contained in a disk centered at $0$ with radius $\n{\bQ}_H$;
  \item[(v)] $\mF_H(\bQ)$ is compact and convex.
\end{enumerate}
\end{lemma}
\begin{proof}Set $\widehat{\bQ} : = \bH^\frac{1}{2}\bQ \bH^{-\frac{1}{2}}$. 
\begin{itemize}
\item[(i)] For any $\bu\in \IC^N\setminus \{\boldsymbol{0}\}$, one can define $\widehat{\bu} : = \bH^\frac{1}{2} \bu $ such that
$$\frac{(\bQ\bu,\bu)_H}{(\bu,\bu)_H} =  \frac{(\bH^\frac{1}{2}\widehat{\bQ}\bH^\frac{1}{2} \bu ,\bu)_2 }{(\bH^\frac{1}{2} \bH^\frac{1}{2} \bu,\bu)_2}= \frac{(\widehat{\bQ} \widehat{\bu},\widehat{\bu})_2}{(\widehat{\bu},\widehat{\bu})_2},
$$
proving that $\mF_H (\bQ)=\mF_2(\widehat{\bQ})$, and that $\n{\widehat{\bQ}}_2 = \n{\bQ}_H$.
\item[(ii)] By \cite[Section 4, Item 1]{benzi2016localization}, one has $\mathfrak{S}(\widehat{\bQ}) \subset \mF_2 ( \widehat{\bQ})$. Clearly, $\bQ$ and $\widehat{\bQ}$ share the same spectrum.
\item[(iii)] If $\bQ$ is $H$-normal, there holds that $\bQ \bH^{-1}\bQ^H \bH = \bH^{-1} \bQ^H \bH \bQ$, hence $$\bH^\frac{1}{2} \bQ \bH^{-1}\bQ^H \bH\bH^{- \frac{1}{2}} = \bH^\frac{1}{2} \bH^{-1} \bQ^H \bH \bQ \bH^{- \frac{1}{2}},$$ leading to $\widehat{\bQ}\widehat{\bQ}^H = \widehat{\bQ}^H \widehat{\bQ}$, proving that $\widehat{\bQ}$ is normal. By \cite[Section 4, Item 10]{benzi2016localization}, we deduce that $\mF_2(\widehat{\bQ}) = \textup{Conv} (\mathfrak{S} (\widehat{\bQ}))$.
\item[(iv)] $\mF_2(\widehat{\bQ})$ is contained in a disk centered at zero with radius $\n{\widehat{\bQ}}_2$ \cite[Section 4, Item 3]{benzi2016localization}. Moreover, $\mF_H (\bQ)= \mF_2(\widehat{\bQ})$ and $ \n{\bQ}_H = \n{\widehat{\bQ}}_2$.
\item[(v)] $\mF_2(\widehat{\bQ})$ is compact and convex by \cite[Section 4, Items 7 and 12]{benzi2016localization}.
\end{itemize}
\end{proof}

\begin{figure}[t]
\centering
\includegraphics[width=.8\textwidth]{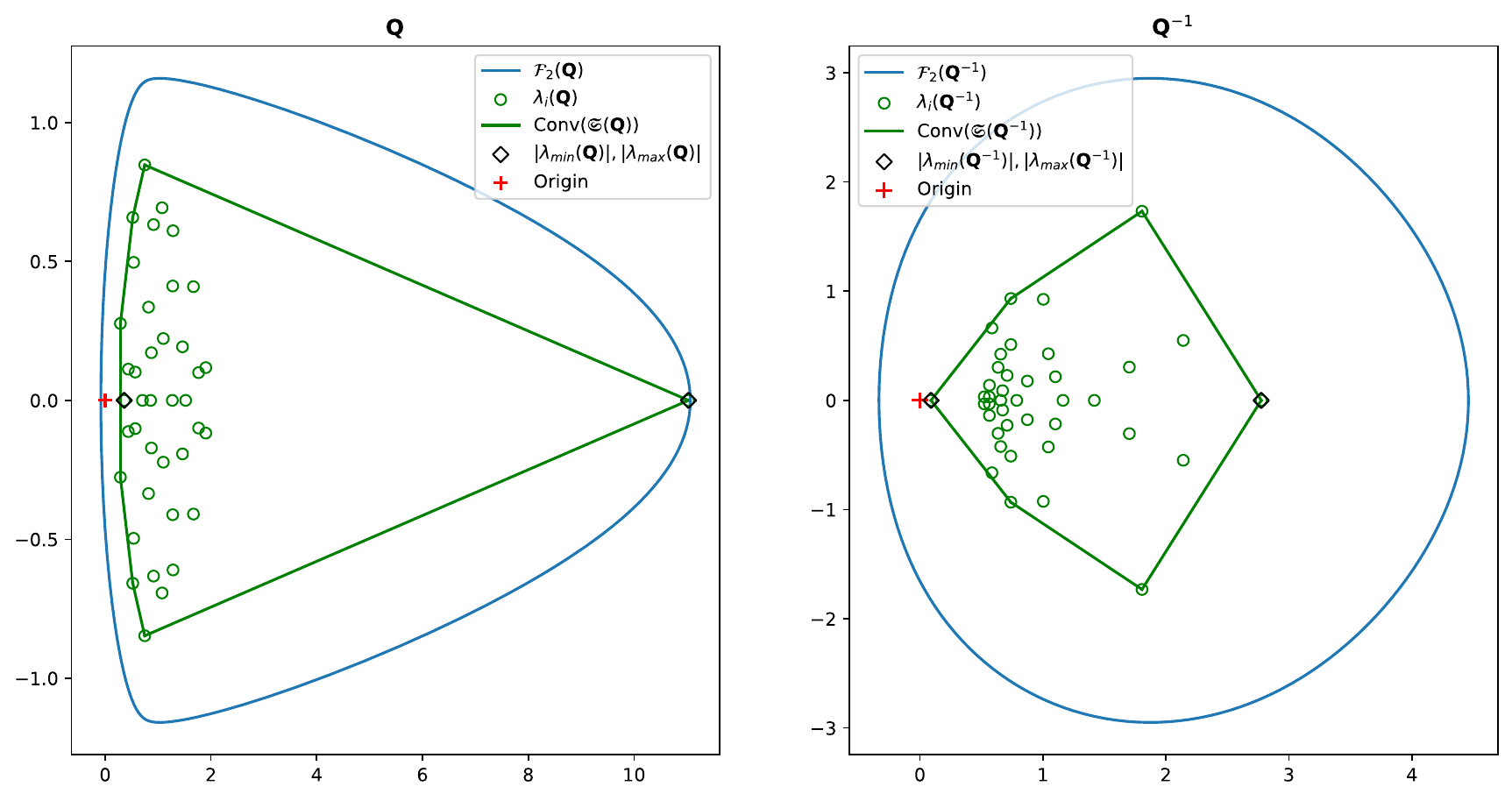}
\caption{$2$-FoV boundary (blue line), eigenvalues (green circles), convex hull for eigenvalues (green line) and $|\lambda_\text{min}|,|\lambda_\text{max}|$ (black diamonds) for a matrix $\bQ:=\bI + 0.5\mathbf{E} \in \IR^{40\times 40}$ (left) and its inverse $\bQ^{-1}$ (right). $\mathbf{E}$ is a random matrix with $\mathbf{E}_{i,j}$ uniformly distributed random numbers in $[0,1]$ for $0\leq i,j\leq 40$. Remark that $0 \neq \textup{Conv}(\mathfrak{S}(\bQ))$ (resp.~$0 \neq \textup{Conv}(\mathfrak{S}(\bQ^{-1}))$) while $0 \in \mF_2(\bQ)$ (resp.~$0 \in \mF_2(\bQ^{-1})$).}
\label{fig:numrange}
\end{figure}

\Cref{fig:numrange} illustrates the above definitions for a random matrix. Remark that: (i) $\bQ$ is invertible, as $|\lambda_\text{min}(\bQ)| >0$; (ii) $\mF_2(\bQ) \not\subset \textup{Conv}(\mathfrak{S}(\bQ))$ and $\mF_2(\bQ^{-1}) \not\subset \textup{Conv}(\mathfrak{S}(\bQ^{-1}))$; (iii) $\text{Conv}(\mathfrak{\mathfrak{S}(\bQ}))$ and $\text{Conv}(\mathfrak{S}(\bQ^{-1}))$ are bounded away from the origin, whereas $0\in \mF_2(\bQ)$ and $0\in \mF_2(\bQ^{-1})$. Moreover, one has $\kappa_S(\bQ) = 30.6$ while $\kappa_2(\bQ)= 58.3$, evidencing the non-normality of $\bQ$.

\subsection{General linear convergence estimates for GMRES$(m)$}\label{subsec:iterative}
Following \cite[Chapter 5]{graham2017domain}, let us recall the application of the weighted (resp.~Euclidean) GMRES to a linear system $\bQ \bx = \bd$ in $\IC^N$, where $\bQ\in\IC^{N\times N}$ is a complex nonsingular matrix. For an initial guess $\bx_0 \neq \bx$, we introduce the residual $\br_0 = \bd - \bQ \bx_0$ such that $\br_0 \neq 0$ as well as Krylov spaces
\be\label{eq:Krylov}
\mK^k (\bQ , \br_0) : = \textup{span} \{  \bQ^j \br_0 : j= 0,\ldots, k-1 \}, \quad  1 \leq k \leq N.
\ee
For any step $1 \leq k \leq N$, we define $\bx_k$ and $\tilde{\bx}_k$ to be the unique elements of $\mK^k(\bQ, \br_0)$ satisfying the minimal residual property under the energy and Euclidean norms:
\be
\begin{split}
\n{\br_k}_H &: = \n{\bd - \bQ \bx_k }_H = \min_{\bx \in \mK^k  (\bQ , \br_0)} \n{\bd - \bQ \bx}_H, \\
\n{\tilde{\br}_k}_2 &: = \n{\bd - \bQ \tilde{\bx}_k }_2 = \min_{\bx \in \mK^k  (\bQ , \br_0)} \n{\bd - \bQ \bx}_2,
\end{split}
\ee
respectively. We refer to either weighted or Euclidean GMRES collectively as GMRES. We add an $(m)$ superscript to signal restarted GMRES$(m)$, for any natural number $1\leq m\leq N$. Notice that GMRES and GMRES$(m)$ coincide up to iteration $m$.

\begin{lemma}[Weighted GMRES$(m$): Linear bounds]\label{lemma:Weighted_GMRES}Let $\bQ\in \IC^{N\times N}$, with $0 \neq \mF_H(\bQ)$ in \eqref{eq:numerical_range} and set $1\leq m\leq N$. Then, the $k$-th residual of weighted GMRES$(m)$ for $1\leq k\leq N$ satisfies:
\be\label{eq:elmanm}
\frac{\n{\br_k}_H}{\n{\br_0}_H} \leq \left(1 -  \mV_H(\bQ) \mV_H\left(\bQ^{-1}\right)  \right)^\frac{k}{2}.
\ee
\end{lemma}
\begin{proof}
We first remark that \Cref{lemma:Weighted_GMRES} for the Euclidean GMRES, i.e.~for $\bH=\bI$, is proved in \cite{liesen2012field}. Thus, we focus on the extension to weighted GMRES. Following \cite[Theorem 5.1]{graham2017domain}, we set $\widehat{\bQ} : = \bH^\frac{1}{2} \bQ \bH^{-\frac{1}{2}}$, $\widehat{\bd}:= \bH^\frac{1}{2} \bd$, $\widehat{\bx}: = \bH^\frac{1}{2}\bx$ and $\widehat{\br}_0 := \bH^\frac{1}{2} \br_0$. Application of the Euclidean GMRES to $\widehat{\bQ}\widehat{\bx} = \widehat{\bd}$ yields 
\be \label{eq:proof_linear}
\frac{\n{\widehat{\br}_k}_2}{\n{\widehat{\br}_0}_2} \leq \left(1 -  \mV_2 (\widehat{\bQ} ) \mV_2(\widehat{\bQ}^{-1})  \right)^\frac{k}{2}.
\ee
By \Cref{lemma:numrange}, we obtain that $\mF_H (\bQ)=\mF_2(\widehat{\bQ})$ and $\mF_H (\bQ^{-1})= \mF_2(\bH^{1/2} \bQ^{-1} \bH^{-1/2}) =\mF_2(\widehat{\bQ}^{-1})$. Consequently, by \eqref{eq:proof_linear} we derive the final bound for the weighted GMRES
\be 
\label{eq:GMRES_weight}
\frac{\n{\br_k}_H}{\n{\br_0}_H} \leq \left(1 -  \mV_H (\bQ) \mV_H(\bQ^{-1})  \right)^\frac{k}{2}=\rho^k,
\ee
with $\rho :=  \left(1 - \mV_H (\bQ) \mV_H(\bQ^{-1})\right)^\frac{1}{2}$ and $\rho < 1$ by \cite[Section 1]{liesen2012field}. Finally, we remark that $\rho$ above does not depend on $k$: it provides a one-step bound. Therefore, we set a restart $1 \leq m \leq N $ and define $k= : i m + t$ with $0 \leq t < m$ and $0 \leq i \leq \lfloor \frac{N}{m}\rfloor$. By application of \eqref{eq:GMRES_weight}, there holds that:
\be 
\n{\br_{im + t}}_H  \leq \rho^t \n{\br_{im}}_H,
\ee 
and thus, $\n{\br_k}_H \leq \rho^k \n{\br_0}_H$, leading to the expected result for GMRES$(m)$.
\end{proof}
\begin{remark}
\Cref{lemma:Weighted_GMRES} provides a linear convergence bound for weighted GMRES$(m)$ with respect to $\mV_H(\bQ)$ and $\mV_H(\bQ^{-1})$ and constitutes a sharper version of the classic result for weighted GMRES in \cite{graham2017domain,sarkis2007optimal}. Indeed, for $\bQ$ in \Cref{lemma:Weighted_GMRES} there holds that \cite[Section 1]{liesen2012field}:
$$ 1 - \mV_H(\bQ) \mV_H(\bQ^{-1}) \leq 1- \frac{\mV_H(\bQ)^2}{\|\bQ\|_H^2} < 1.$$
\end{remark}
\subsection{Discrete $\langle X_h\rangle$- and $(X_h)$-coercivity}
For the ensuing GMRES analysis of our preconditioned problem \CA, we need precise definitions for coercivity that relate to the BNB condition for Hilbert spaces. We introduce the notion of (discrete) $\langle X_h\rangle$-coercivity, with angle brackets referring to the dual pairing in \eqref{eq:dual_pairing} (refer to \cite[Section 5]{chandler2012numerical}). 

\begin{definition}[$\langle X_h\rangle$-coercivity]\label{def:discrete_coercivity_dual}Consider $\fA:X \to X'$ as in the BG case with $X$ being a Hilbert space. For $h>0$ given, $\fA$ is said to be $\langle X_h \rangle $-coercive if there exists $\alpha_\fA$ such that 
\be 
0 <\alpha_\fA \leq\frac{|\fa(u_h,u_h)|}{\n{u_h}_X^2} = \frac{|\langle \fA u_h,u_h\rangle_{X'\times X}|}{\n{u_h}_X^2} \quad \forall \ u_h \in X_h  \setminus \{\boldsymbol{0}\}.
\ee
\end{definition}
Thus, discrete $\langle X_h \rangle$-\emph{ellipticity} refers to self-adjoint operators satisfying the $\langle X_h\rangle$-coercivity condition. These definitions extend naturally to continuous $\langle X\rangle$-coercivity and -ellipticity.
\begin{remark}[BNB condition and $\langle X_h \rangle$-coercivity]\label{rmk:inf-sup}
As pointed out for $\langle X \rangle$-coercivity in \cite[Lemma 2.8]{ern2004theory} , $\langle X_h \rangle$-coercivity for $\fA$ in \Cref{def:discrete_coercivity_dual} provides a BNB constant $\gamma_\fA = \alpha_\fA$, since for any $u_h \in X_h \setminus \{\boldsymbol{0}\}$, it holds that
\be 
0 < \alpha_\fA \n{u_h}_X \leq \frac{ |\fa(u_h,u_h)|}{\n{u_h}_X} \leq \sup_{v_h \in X_h \setminus \{\boldsymbol{0}\}} \frac{ |\fa(u_h,v_h)|}{\n{v_h}_X}.
\ee
\end{remark}
\begin{remark}
$\langle X_h \rangle$-coercivity is a common property for BG methods in Hilbert spaces. For example, under suitable assumptions on the discretization scheme, operators with a G\aa rding inequality on $X$---of the form $\fA = \fA_0 +\fK:X \to X'$ with $\fA_0$ $\langle X \rangle$-coercive and $\fK$ compact \cite[Section 2.1]{sauter2010boundary}--- admit a $h_0>0$ such that $\fA$ is $\langle X_h \rangle$-coercive for all $0 < h \leq h_0$ \cite[Section 4.2.3]{sauter2010boundary}. 
\end{remark}

Similarly to \Cref{def:discrete_coercivity_dual}, we introduce the discrete $(X_h)$-coercivity, the difference being the use of inner products.
\begin{definition}[$(X_h)$-coercivity]\label{def:H_coercivity}Consider $\fA:X\to X$ for the PGE case with $X=: H$ being a Hilbert space. $\fA$ is said to be $(X_h)$-coercive if, for any $u_h \in X_h \setminus \{\boldsymbol{0}\}$, there exists $\alpha_\fA >0$ such that 
\be 
\alpha_\fA  \leq  \frac{ \left|(\fA u_h ,u_h)_H\right| }{(u_h,u_h)_H},
\ee
or equivalently 
\be\label{eq:HFOVcoercivity}
\alpha_\fA  \leq  \inf_{u_h \in X_h \setminus \{\boldsymbol{0}\}} \frac{ \left|(\fA u_h ,u_h)_H\right| }{(u_h,u_h)_H}= \mV_H(\fA_h).
\ee
\end{definition}
\Cref{def:H_coercivity} via \eqref{eq:HFOVcoercivity} shows the strong connection between $(X_h)$-coercivity and discrete $\mV_H$. Furthermore, under the OP setting, the discrete and matrix $H$-FoVs coincide.
\begin{lemma}\label{lemma:coicideFoV}Consider \CA~with $X=:H$ a Hilbert space with inner product $(\cdot,\cdot)_H$. There holds that
\begin{enumerate}
\item[(i)]  $\mF_H(\fP_h\fA_h) = \mF_H (\bP \bA)$ and $\mF_H((\fP_h\fA_h)^{-1}) = \mF_H ((\bP \bA)^{-1})$;
\item[(ii)] $\mV_H(\fP_h\fA_h) = \mV_H (\bP \bA)$ and $\mV_H((\fP_h\fA_h)^{-1}) = \mV_H ((\bP \bA)^{-1})$.
\end{enumerate}
\end{lemma}
\begin{proof}Following \eqref{eq:OPalgebra}, $q_h :=\fP_h\fA_h u_h \in X_h$ for any $u_h \in X_h$ has basis expansion $\bq =\bP \bA \bu$. Therefore, there holds by \eqref{eq:def_inner_product} that for any $u_h \in X_h$:  
\be 
(\fP_h \fA_h u_h , u_h)_H = (\bP \bA \bu , \bu)_H\quad \text{and} \quad (u_h,u_h)_H = (\bu, \bu)_H,
\ee
yielding $\mF_H(\fP_h\fA_h)= \mF_H(\bP\bA)$, and thus the expected result for $\fP_h \fA_h$. Similarly, one deduces the same result for the inverse operator since for any $q_h \in X_h$, $(\fP_h\fA_h)^{-1}q_h$ has basis expansion $(\bP \bA)^{-1} \bq$.\end{proof}

\subsection{Linear convergence estimates for GMRES$(m)$ applied to \CA}\label{subsec:LinearCA}

Following \Cref{subsec:iterative}, application of the weighted (resp.~Euclidean) preconditioned GMRES$(m)$, for $1 \leq m \leq N$, to \CA$_{\mu,\nu}$ and initial guess $\bx_0 \neq \bu_\nu$, \rev{produces} the iterates $\bx_k$ (resp.~$\tilde{\bx}_k$) for any step $1\leq k \leq N$ with minimal residual properties:
\be\label{eq:prec_residual}
\begin{split}
\n{\bP_\mu \br_k}_H &: = \n{\bP_\mu \bb_\nu - \bP_\mu \bA_\nu \bx_k }_H = \min_{\bx \in \mK^k  (\bP_\mu \bA_\nu , \br_0)} \n{\bP_\mu \bb_\nu - \bP_\mu \bA_\nu \bx}_H, \\
\n{\bP_\mu \tilde{\br}_k}_2 &: = \n{\bP_\mu \bb_\nu - \bP_\mu \bA_\nu \tilde{\bx}_k }_2 = \min_{\bx \in \mK^k \left(\bP_\mu \bA_\nu , \br_0\right)} \n{\bP_\mu \bb_\nu - \bP_\mu \bA_\nu \bx}_2,
\end{split}
\ee
with $\mK^k$ introduced in \eqref{eq:Krylov} and obvious construction for \CA. By \eqref{eq:prec_residual}, the minimal residuals satisfy
\be\label{eq:minimal}
\n{\bP_\mu \tilde{\br}_k}_2 \leq \n{\bP_\mu \br_k}_2\quad \text{and} \quad  \n{\bP_\mu \br_k}_H \leq \n{\bP_\mu \tilde{\br}_k}_H.
\ee
We set convergence rates for the weighted (resp.~Euclidean) preconditioned GMRES$(m)$:
\be\label{eq:ConvRates}
\Theta^{(m)}_k := \left(\frac{\n{\bP_\mu \br_k}_H}{ \n{\bP_\mu \br_0}_H}\right)^\frac{1}{k} \quad \text{and} \quad  \widetilde{\Theta}_k^{(m)} :=\left( \frac{\n{\bP_\mu \tilde{\br}_k}_2}{ \n{\bP_\mu \br_0}_2}\right)^\frac{1}{k}.
\ee
Finally, we define convergence rates for non-restarted weighted (resp.~Euclidean) preconditioned GMRES $\Theta_k :=\Theta_k^{(N)}$ (resp.~$\widetilde{\Theta}_k:=\widetilde{\Theta}_k^{(N)}$).

The bounds in \Cref{thm:opprecond} and \Cref{thm:strang} are related to the spectral radius, and rely on the continuity and discrete inf-sup constants: they do not supply information on the eigenvalue or FoV distributions, as pointed out in \Cref{fig:numrange}, which are required to derive convergence results for iterative solvers to \CA or \CA$_{\mu,\nu}$. Thus, more specific conditions are required. For instance, for \CA~we enforce the following $(X_h)$-coercivity condition for $\fP_h\fA_h$. 

\begin{assumption}[$(X_h)$-coercivity for \CA]\label{ass:FOV}For problem \CA~with $X:=H$ being a Hilbert space with inner product $(\cdot,\cdot)_H$, we assume that $\fP_h\fA_h$ and its inverse are $(X_h)$-coercive satisfying
\be \label{eq:def_CAL}
\frac{\gamma_\fC \gamma_\fA}{\|\fm\|\|\fn\|}\leq \mV_H(\fP_h\fA_h) \quad \text{and}  \quad \frac{\gamma_\fM \gamma_\fN}{\|\fc\|\|\fa\|}  \leq \mV_H((\fP_h \fA_h)^{-1}).
\ee
\end{assumption}
\begin{remark}\label{rmk:extensionXh}The $(X_h)$-coercivity constants in \Cref{ass:FOV} emerge naturally, as they are related to the BNB constants for both $\fP_h\fA_h$ and its inverse (cf.~proof of \Cref{thm:opprecond}). Alternatively, one can write \Cref{ass:FOV} as:
\be 
0 < \Gamma_0 \leq \mV_H(\fP_h \fA_h) \quad \text{and}\quad 0<  \Gamma_1 \leq \mV_H((\fP_h\fA_h)^{-1}),
\ee 
with $\Gamma_0, \Gamma_1$ constants depending on the discrete inf-sup and continuity constants for the induced operators and eventually for $(\cdot, \cdot)_H$ in \eqref{eq:def_inner_product}.
\end{remark}

We are ready to state the linear convergence result for GMRES$(m)$ for \CA~for the Hilbertian case.
\begin{theorem}[{GMRES$(m)$: Linear convergence estimates for \CA}]
\label{lemma_linear}
Consider \CA~with $X=:H$ Hilbert and $(\cdot,\cdot)_H$ such that \Cref{ass:FOV} holds. Then, GMRES($m$) for $1\leq k,m \leq N$ leads to
\be
\label{eq:elmanCA}
\Theta_k^{(m)}  \leq  \left(1- \frac{1}{\K_\star}\right)^\frac{1}{2} \quad \textup{and} \quad \widetilde{\Theta}_k^{(m)} \leq  \KL\left(1 - \frac{1 }{\K_\star}\right)^\frac{1}{2},
\ee
with $\K_\star$ as defined in \eqref{eq:kS_star} and $\KL$ in \eqref{eq:KL}.
\end{theorem}
\begin{proof}By combining \Cref{ass:FOV}, \Cref{lemma:coicideFoV} and definition of $\K_\star$, there holds that
\be 
\mV_H (\bP\bA) \mV_H((\bP\bA)^{-1}) \geq \frac{1}{\K_\star} 
\ee
and thus
$$1 -\mV_H (\bP\bA) \mV_H((\bP\bA)^{-1}) \leq 1 - \frac{1}{\K_\star}$$
Application of \Cref{lemma:Weighted_GMRES} to the preconditioned system with residuals \eqref{eq:prec_residual} provides the first bound in \eqref{eq:elmanCA}, namely
\be\label{eq:proof_lemma_linear1}
\frac{\n{\bP \br_k}_H}{\n{\bP \br_0}_H} \leq \left(1 - \frac{1}{\K_\star}  \right)^\frac{k}{2}, \quad 1 \leq k \leq N.
\ee
Next, we follow the steps in \cite[Section 4]{sarkis2007optimal} to arrive at the second bound in \eqref{eq:elmanCA}. First, the minimal residual property \eqref{eq:minimal} yields
\be\label{eq:proof_lemma_linear2}  
\n{\bP \tilde{\br}_k}_2\leq \n{\bP \br_k}_2. 
\ee
By virtue of the synthesis operator in \eqref{eq:synthesis}, one has
\be\label{eq:proof_lemma_linear3} 
\n{\bP \br_k}_2 \leq  \frac{1}{\gamma_{\Lambda_h}} \n{\bP \br_k}_H\quad \text{and} \quad \n{\bP \br_0}_H \leq \|\Lambda_h\| \n{\bP \br_0}_2.
\ee
Therefore, \eqref{eq:proof_lemma_linear1} combined with \eqref{eq:proof_lemma_linear2} and \eqref{eq:proof_lemma_linear3} lead to the final result, as
\begin{align*}
\n{\bP \tilde{\br}_k}_2\leq \n{\bP \br_k}_2 \leq  \frac{1}{\gamma_{\Lambda_h}}\n{\bP \br_k}_H\leq  \frac{1}{\gamma_{\Lambda_h}}\left(1- \frac{1}{\K_\star}\right)^\frac{k}{2} \n{\bP \br_0}_H 
\leq  \KL\left(1- \frac{1}{\K_\star}\right)^\frac{k}{2} \n{\bP \br_0}_2,
\end{align*}
as stated.
\end{proof}
\begin{remark}
This result provides extensive convergence bounds for GMRES$(m)$. It will guarantee $h$-independent convergence for weighted GMRES$(m)$ for \CA~in the Hilbert setting (cf.~\Cref{remark:Asymptotics}). Also, the synthesis operator enters as an offset factor in $\widetilde{\Theta}_k^{(m)}= \KL \rho$ with $\rho:= (1- 1/ \K_\star)^{1/2} <1$. One should observe that $\widetilde{\Theta}_k^{(\text{m})}$ could be larger than 1 for $\KL >1$, an impractical bound for Euclidean GMRES$(m)$. The latter supports theoretically the use of weighted GMRES$(m)$ as a solver \cite{feischl2017optimal}.
\end{remark}
To illustrate the application of the above results, we provide a case of interest where \Cref{ass:FOV} is satisfied. Therein, notice the extra $\fK_\fA$-term in \eqref{eq:lemma_precinduced} below, justifying \Cref{rmk:extensionXh}. 
\begin{corollary}[Preconditioner-induced norm {\cite{feischl2017optimal,starke1997field,kirby2010functional}}]\label{case:case}Consider \CA~for OP-BG for Hilbert spaces $X=:H$ and $V$, $\fA$ being $\langle X_h \rangle$-coercive, and $\fC$ being $\langle V_h \rangle$-elliptic, with $\gamma_\fA := \alpha_\fA$ and $\gamma_\fC:=\alpha_\fC$. Then, $\fP^{-1}$ is Hermitian and yields an inner product on $X_h$, denoted by $(\cdot,\cdot)_{P^{-1}}$, and
\be \label{eq:lemma_precinduced}
 \frac{\gamma_\fC\gamma_\fA }{\|\fm\|^2}  \leq  \mV_{P^{-1}}(\fP_h \fA_h) \quad\text{and}\quad \frac{\gamma_\fM^2}{\|\fc\|\|\fa\|}\K_\fA \leq \mV_{P^{-1}}((\fP_h \fA_h)^{-1}).
\ee
\end{corollary}\begin{proof}First, notice that since $\fC$ is $\langle V_h\rangle$-elliptic, $\bC$ and $\bC^{-1}$ are Hermitian positive definite. Therefore, we deduce that $\bP= \bM^{-1}\bC \bM^{-H} = \bP^H$ and $\bP^{-1} = \bP^{-H}$ are Hermitian positive definite. For any $\bu,\bv\in \IC^N\setminus \{\boldsymbol{0}\}$ and $\bw := \bP^{-1}\bv$, one has
$$
\mV_{P^{-1}} ( \fP_h \fA_h) =  \inf_{\bu \in \IC^N \setminus \{\boldsymbol{0}\}}\frac{|(\bP \bA \bu , \bu )_{P^{-1}}|}{(\bu,\bu)_{P^{-1}}} =   \inf_{\bu \in \IC^N \setminus \{\boldsymbol{0}\}}\frac{| (\bA \bu, \bu)_2|}{|(\bP^{-1} \bu,\bu)_2|}
$$
and 
$$
\quad\mV_{P^{-1}} ( (\fP_h \fA_h)^{-1}) =   \inf_{\bv \in \IC^N \setminus \{\boldsymbol{0}\}}\frac{|(\bA^{-1}\bP^{-1} \bv , \bv )_{P^{-1}}|}{(\bv,\bv)_{P^{-1}}} =  \inf_{\bw \in \IC^N \setminus \{\boldsymbol{0}\}}\frac{| (\bA^{-1} \bw, \bw)_2|}{|(\bP \bw,\bw)_2|}.
$$
Next, using Equations 2.61 and 2.62 in Kirby \cite{kirby2010functional}, we deduce that:
\be 
\gamma_\fA \|u_h\|_X^2 \leq |(\bA \bu,\bu)_2| \leq \|\fa\|\|u_h\|_X^2\quad \text{and} \quad \frac{\gamma_\fA}{\|\fa\|} \|w_h\|_{X_{{h}}'}^2 \leq |(\bA^{-1} \bw,\bw)_2| \leq \|\fa\|\|w_h\|_{X_{{h}}'}^2,
\ee 
while for the preconditioner $\bP$ one has by \cite[Section 13.2]{steinbach2007numerical} that
\be 
\frac{\gamma_\fC}{\|\fm\|^2} \|w_h\|_{X_h'}^2 \leq |(\bP \bw,\bw)_2| \leq \frac{\|\fc\|}{\gamma_\fM^2}\|w_h\|_{X_{{h}}'}^2 \quad \text{and} \quad \frac{\gamma_\fM^2}{\|\fc\|} \|u_h\|_{X}^2 \leq |(\bP^{-1} \bu,\bu)_2| \leq \frac{\|\fm\|^2}{\gamma_\fC }\|u_h\|_{X}^2.
\ee 
Therefore, 
\be 
\frac{\gamma_\fA \gamma_\fC}{\|\fm\|^2} \leq \frac{| (\bA \bu, \bu)_2|}{|(\bP^{-1} \bu,\bu)_2|} \quad \text{and} \quad \frac{\gamma_\fA \gamma_\fM^2}{\|\fa\|^2 \|\fc\|}  \leq \frac{| (\bA^{-1} \bw, \bw)_2|}{|(\bP \bw,\bw)_2|},
\ee 
finalizing the proof.
\end{proof}
\begin{corollary}\label{colo:case}Consider \CA~for OP-BG for Hilbert spaces $X=:H$ and $V$, $\fA$ being $\langle X_h \rangle$-coercive, and $\fC$ being $\langle V_h \rangle$-elliptic. Then, GMRES$(m)$ for $1\leq k , m \leq N$ leads to
\be 
\Theta_k^{(m)} \leq \left(1 - \frac{1}{\K_\star \K_\fA}\right)\quad \text{and} \quad \widetilde{\Theta}_k^{(m)} \leq  \KL\left(1 - \frac{1 }{\K_\star \K_\fA}\right)^\frac{1}{2},
\ee
with $\K_\star$ as defined in \eqref{eq:kS_star}, $\KL$ in \eqref{eq:KL} and $\K_\fA$ in \eqref{eq:discrete_condition_number}.
\end{corollary}
\begin{remark}
The $\K_\fA$-term in \Cref{case:case} and \Cref{colo:case} is removed if $\fA$ is $\langle X_h \rangle$-elliptic, since $\fA^{-1}$ is $\langle X_h'\rangle$-elliptic with constant $1/\|\fa\|$ \cite[Section 13.2]{steinbach2007numerical}.
\end{remark}

\subsection{Linear convergence estimates for GMRES$(m)$ applied to \CA$_{\mu,\nu}$}

As in \Cref{subsec:LinearCA}, we give counterparts to \Cref{ass:FOV} and \Cref{lemma_linear} for the bi-parametric preconditioned problem \CA$_{\mu,\nu}$. 
\begin{assumption}[$(X_h)$-coercivity for \CA$_{\mu,\nu}$]\label{ass:FOVmn}For \CA$_{\mu,\nu}$~with $X:=H$ being a Hilbert space with inner product $(\cdot,\cdot)_H$, assume that there holds that
\be \label{eq:def_CALmn}
\frac{\gamma_{\fC_\mu} \gamma_{\fA_\nu}}{\|\fm\|\|\fn\|}\leq \mV_H(\fP_{h,\mu} \fA_{h,\nu}) \quad \text{and}  \quad \frac{\gamma_\fM \gamma_\fN}{\|\fc_\mu\|\|\fa_\nu\|} \leq \mV_H((\fP_{h,\mu} \fA_{h,\nu})^{-1}).
\ee
\end{assumption}

\Cref{rmk:extensionXh} remains valid for \Cref{ass:FOVmn}. With this, we can extend \Cref{lemma_linear} to \CA$_{\mu,\nu}$.

\begin{theorem}[GMRES$(m)$: Linear convergence estimates for \CA$_{\mu,\nu}$]
\label{lemma_linear_strang}
Consider \CA$_{\mu,\nu}$~along with \Cref{ass:FOVmn}. Then, the residuals for GMRES($m$) for $1\leq k,m \leq N$ are bounded as
\be
\label{eq:linear_CAmn}
\Theta_k^{(m)}  \leq  \left(1- \frac{1}{\K_{\star,\mu,\nu} } \right)^\frac{1}{2} \quad \textup{and} \quad \widetilde{\Theta}_k^{(m)} \leq  \KL \left(1 - \frac{1 }{\K_{\star,\mu,\nu}}\right)^\frac{1}{2},
\ee
with $\K_{\star,\mu,\nu}$ and $\KL$ defined in \eqref{eq:kS_mu_nu} and \eqref{eq:KL}, respectively. 
\end{theorem}
\begin{proof}The result follows by direct application of \Cref{lemma_linear} to \CA$_{\mu,\nu}$.
\end{proof}
\begin{remark}
The above result gives a controlled convergence rate for GMRES$(m)$ with respect to bi-parametric $(\mu,\nu)$-perturbations. As in the discussion ensuing \Cref{thm:strang} and in order to illustrate its practical implications, assume that the best approximation error in \Cref{lemma:strang} converges at a rate $\mO(h^r)$, $r>0$, and $\nu = \mO(h^r)$. Therefore, provided that $\mu=\mO(1)$ guarantees a bounded $\K_{\star,\mu,\nu}$, the bounds in \eqref{eq:linear_CAmn} ensure linear convergence for the weighted GMRES$(m)$ (resp.~Euclidean GMRES$(m)$, for $\KL < 1$).
\end{remark}

\subsection{Compact and Carleman class operators }\label{subsec:Carleman}

So far, we have focused on the linear convergence rates for GMRES($m$). Yet, it is know that in many situations the bound in \Cref{lemma:Weighted_GMRES} ``\emph{may significantly overestimate the GMRES residual norms}" \cite{liesen2012field}. To better understand this, we aim to improve bounds for the case of second-kind Fredholm operators, which are known to display super-linear convergence results for GMRES, i.e.~the radius of convergence tends to zero as $k\to \infty$. To this end, we introduce the concept of Carleman class operators.

Again, assuming $H$ to be a separable Hilbert space, we introduce $\mC(H)\equiv \mC(H;H)$ the space of compact operators on $H$. Given $\fT\in \mL(H;H)$, we denote the ordered singular values of $\fT$ as $\sigma_j(\fT) := \{\inf {\|\fT-\fT_i\|}_H : \fT_i : H\to H , \ \textup{rank} \fT_i < j\}$. For any $k\geq 1$, the $k$th partial arithmetic mean for the singular values reads
\be\label{eq:partial_singular_values}
\overline{\sigma}_k (\fK): = \frac{1}{k} \sum_{j=1}^k \sigma_j (\fK).
\ee
For $p>0$, a compact operator $\fK \in\mC(H)$ is said to belong to the \emph{Carleman class} $\mC^p(H)$ \cite[Section XI.9]{dunford1963linear} if it holds that
\be\label{eq:Carleman} 
\iii{\fK}_p = \|\sigma(\fK)\|_{p} := \left(\sum_{i=1}^\infty {\sigma_i(\fK)}^p \right)^{1/p}<\infty. 
\ee
Next, we say that $\fQ$ is a $p$\emph{-class Fredholm operator of the second-kind}, $\fQ \in \mF\mC^p(H)$ if and only if $\fQ - \fI \in \mC(H)$ for $p=0$ \rev{(resp.~$\fQ - \fI \in \mC^p(H)$ for $p>0$)}.
Consequently, for $H =: X$ a separable Hilbert space, $p\geq 0$ and $\nu, \mu\in [0,1)$, we define the following problems:
\be \label{eq:Ap}
\A^p : \quad \A\quad \text{for PGE (i.e.~$\fA : H\to H$) with } \fA \in \mF\mC^p(H) \quad \text{and}\quad \fN: = \fI,
\ee
and
\be\label{eq:CAmnp} 
\CA_{\mu,\nu}^p : \quad \CA_{\mu,\nu}\quad \text{with}\quad  \fC_\mu \fN^{-1} \fA_\nu \in \mF\mC^p(H),
\quad \text{and} \quad \fM: = \fI, 
\ee
whose diagram representation is
$$
\textup{\CA}_{\mu,\nu}^p: \quad\begin{tikzcd}
H \arrow{r}{\fA_\nu}  & Y' \arrow{d}{\fN^{-1}} \\
H \arrow{u}{\fI^{-1}} & V\arrow{l}{\fC_\mu} 
\end{tikzcd}.
$$
Finally, for \A$^p$ and \CA$^p_{\mu,\nu}$, the corresponding compact terms $\fK:=\fA -\fI $ and $\fK_{\mu,\nu}: = \fC_\mu \fN^{-1}\fA_\nu - \fI$ have discrete counterparts $\fK_h := \fA_h-\fI_h$ and $\fK_{h,\mu,\nu}:= \fC_h \fN_h^{-1}\fA_h-\fI_h$ with Galerkin matrices defined as
\be\label{eq:Kmatrix}
\bK:= \bA - \bN \quad \text{and}\quad \bK_{\mu,\nu} := \bC_\mu \bN^{-1}\bA_\nu-\bM,
\ee 
respectively. In the sequel, we introduce ordered (matrix) singular values with respect to the $H$-norm \cite[Proposition 4.2]{axelssonsuperlinear}:
\be 
\sigma_j^H(\bQ) := \lambda_j (\bQ^\star \bQ)^{1/2} = \sigma_j( \bH^{1/2} \bQ \bH^{-1/2}),
\ee 
for any $\bQ \in \IC^{N\times N}$ and $\bQ^\star= \bH^{-1} \bQ^H \bH$ its $H$-adjoint.
\subsection{Super-linear convergence estimates for GMRES applied to \A$^p$}\label{subsec:second-kind} 
We recall the classic super-linear convergence result for weighted GMRES on a (continuous) Hilbert setting level (cf.~\cite{moret1997note} and \cite[Theorem 3.1]{axelssonsuperlinear}).

\begin{proposition}[Weighted GMRES: Classic super-linear convergence estimate {\cite[Theorem 3.1]{axelssonsuperlinear}}]\label{prop:classic_superlinear}Let $H$ be a Hilbert space. Set $p \geq 0$ and consider the application of weighted GMRES on $\fQ x = {f}$, for a bounded and invertible operator $\fQ \in \mF \mC^p (H)$ with $f \in H$. Introduce GMRES iterates $x_0 \neq x$, and $x_k$, along with $r_k : = \fQ x_k - f$, for any $k \geq 1$. Then, the residuals satisfy
$$
\left(\frac{\n{r_k}_{H}}{\n{r_0}_{H}}\right)^\frac{1}{k} \leq\n{\fQ^{-1}}_{H} \overline{\sigma}_k(\fK), 
$$
wherein \rev{$\fK := \fQ - \fI\in \mC(H)$ for $p=0$ (resp.~ $\mC^p(H)$ for $p>0$)} and $\overline{\sigma}_k(\fK)$ defined in \rev{\eqref{eq:partial_singular_values}}.
\end{proposition}
Remark that $\overline{\sigma}_k(\fK)\to 0$ as $k\to \infty$ evidencing the super-linear convergence rate for residuals of weighted GMRES in this particular case. Furthermore, the convergence rate depends directly on the singular values of the continuous operator $\fK$. The following result shows that the above is applicable to \A$^p$ as well.
\begin{theorem}[GMRES: Super-linear convergence estimates for \A$^p$]
\label{thm:superlinear_sf}
Consider the PGE problem \A$^{p}$ in \eqref{eq:Ap} for any $p \geq0$. Then, for $1 \leq k \leq N$, it holds that
\be\label{eq:sigmamean}
\left(\frac{\n{\br_k}_H}{\n{\br_0}_H}\right)^\frac{1}{k} \leq \frac{ \overline{\sigma}_k(\fK)}{\gamma_\fA \gamma_\fN} \quad \left({\leq \frac{\iii{\fK}_p}{\gamma_\fA \gamma_\fN}  k^{-\frac{1}{p}}}\quad \text{if}\quad p > 0 \right),
\ee 
and 
\be\label{eq:sigmameanEucl}
\left(\frac{\n{\tilde{\br}_k}_2}{\n{\br_0}_2}\right)^\frac{1}{k} \leq \KL  \frac{\overline{\sigma}_k(\fK)}{\gamma_\fA \gamma_\fN} \quad \left({\leq \KL\frac{\iii{\fK}_p}{\gamma_\fA \gamma_\fN}}  k^{-\frac{1}{p}}\quad \text{if}\quad p > 0\right),
\ee
wherein \rev{$\fK:=\fA-\fI \in \mC(H)$ for $p=0$ (resp.~$\mC^p(H)$ for $p>0$)} and $\overline{\sigma}_k(\fK)$ \rev{from \eqref{eq:partial_singular_values}}.
\end{theorem}
\begin{proof} 
By hypothesis, we have that $\bN^{-1} \bA = \bI + \bN^{-1} \bK$, with $\bK$ such as in \eqref{eq:Kmatrix}. Following the same steps as in Axelsson \cite{axelssonsuperlinear}, we deduce that the following relations hold (cf.~proofs of \Cref{thm:opprecond} and \Cref{lemma:coicideFoV}):
$$
\n{(\bN^{-1}\bA)^{-1}}_H \leq \frac{\|\fn\|}{\gamma_\fA}\ = \frac{1}{\gamma_\fA},
$$
since $\fN = \fI$. Furthermore, it holds that the singular values \cite[Proposition 4.2]{axelssonsuperlinear}
$$\label{eq:ineq_sing}
\sigma_j^{{H}}(\bN^{-1} \bK) \leq \frac{1}{\gamma_\fN} \sigma_j(\fK_h) \leq \frac{1}{\gamma_\fN} \sigma_j(\fK).
$$
Therefore, for $1\leq k \leq N$, following \cite{axelssonsuperlinear} and using \Cref{prop:classic_superlinear}, we can show that
\begin{align*}
\frac{ \n{\br_k}_H}{\n{\br_0}_H} &\leq \frac{\n{(\bN^{-1}\bA)^{-1}}_H}{k} \sum_{j=1}^k \sigma_j^{{H}}(\bN^{-1}\bK)\leq \frac{1}{\gamma_\fA} \sum_{j=1}^k \frac{\sigma_j^{{H}}(\bN^{-1}\bK)}{k} \leq\frac{1}{\gamma_\fA\gamma_\fN} \sum_{j=1}^k \frac{\sigma_j(\fK)}{k}.
\end{align*}
Now, if $\fK \in \mC^p(H)$ for any $p>0$, we follow \cite[Theorem 2.2]{winther1980some} and derive
\begin{align*}
\sum_{j=1}^k\frac{\sigma_j(\fK)}{k}\leq \iii{\fK}_p k^{-\frac{1}{p}},
\end{align*}
providing the final estimate in energy norm.

Finally, the bounds in Euclidean norm are deduced in the same fashion as in \Cref{lemma_linear}.
\end{proof}
\begin{remark}
This result appears to be new and it justifies the positive results of employing mass matrix preconditioning, i.e.~$\fN := \fI$, to transfer the super-linear convergence bounds from the continuous to the discrete level. Indeed, the choice of $\fN= \fI$ guarantees a discrete system $\bN^{-1} \bA = \bI + \bN^{-1} \bK$ of the form $\bI$ plus discretization of a compact operator. The latter enables the application of the classical super-linear results for GMRES given in \Cref{prop:classic_superlinear}. Notice that the bounds in \eqref{eq:sigmamean} and \eqref{eq:sigmameanEucl} depend on $k$ via $\overline{\sigma}_k(\fK)$: they are not one-step bounds, and do not generalize to GMRES$(m)$, as the relative error at iteration $k$ for $2\leq k \leq N$ depends on previous iterations.
\end{remark}

\begin{remark}
The super-linear convergence rate depends on the decay rate of $\overline{\sigma}_k(\fK)$. For example, for trace class operators ($p=1$), it holds that $\n{\br_k}_H / \n{\br_0}_H = \mO(k^{-1})$ while for Hilbert-Schmidt operators ($p=2$), one observes the faster rate $\n{\br_k}_H / \n{\br_0}_H = \mO(k^{-2})$ \cite[Chapter XI]{dunford1963linear}. Results describing the Carleman class index for pseudo-differential operators (resp.~the Laplace double-layer operator) can be found in \cite{SOBOLEV20145886} (resp.~\cite{bessoud2006q,miyanishi2015}) and will be investigated elsewhere.
\end{remark}

\subsection{Super-linear convergence estimates for GMRES applied to \CA$_{\mu,\nu}^p$}

We next show that the reasoning in \Cref{thm:superlinear_sf} can also be applied to \CA$^p_{\mu,\nu}$.
\begin{theorem}[GMRES: Super-linear convergence estimates for \CA$_{\mu,\nu}^p$]\label{thm:superMN}
Consider {\CA$^p_{\mu,\nu}$ in \eqref{eq:CAmnp} for any $p\geq 0$} and define $\fK_{\mu,\nu}:=\fC_\mu \fN^{-1} \fA_\nu -\fI \in \mC(H)$ \rev{for $p=0$ (resp.~$\mC^p(H)$ for $p>0$)}. Then, for weighted and Euclidean GMRES, respectively, it holds that
\be\label{eq:Extended_Theta0}
\begin{split}
\Theta_k &\leq \frac{\|\fn\|}{ \gamma_\fC \gamma_\fA\gamma_\fM}\frac{ \overline{\sigma}_k(\fK_{\mu,\nu})}{(1-\mu)(1-\nu)} \quad\left(
  \leq \frac{\|\fn\| }{ \gamma_\fC \gamma_\fA\gamma_\fM}\frac{\iii{\fK_{\mu,\nu}}_p}{(1-\mu)(1-\nu)} k^{-\frac{1}{p}}\quad \text{if} \quad p>0\right),
\end{split}
\ee
and
\be\label{eq:Extended_Thetap}
\begin{split}
\widetilde{\Theta}_k&\leq \KL \frac{\|\fn\|}{ \gamma_\fC \gamma_\fA\gamma_\fM}\frac{\overline{\sigma}_k(\fK_{\mu,\nu})}{(1-\mu)(1-\nu)}\quad
\left(\leq \KL \frac{\|\fn\|}{ \gamma_\fC \gamma_\fA\gamma_\fM}\frac{\iii{\fK_{\mu,\nu}}_p}{(1-\mu)(1-\nu)} k^{-\frac{1}{p}}\quad \text{if} \quad p>0\right),
\end{split}
\ee
with ${\Theta}_k$ and $\widetilde{\Theta}_k$ defined in \eqref{eq:ConvRates} and $\overline{\sigma}_k(\cdot)$ in \rev{\eqref{eq:partial_singular_values}}.
\end{theorem}
\begin{proof}
Consider \CA$^p_{\mu,\nu}$ and follow the proof of \Cref{thm:superlinear_sf}. First, we use \Cref{prop:perturbed} to deduce that
$$
\n{(\bP_\mu \bA_\nu)^{-1}}_H ={ \n{\bA^{-1}_\nu \bN \bC^{-1}_\mu \bM}_H} \leq \frac{\|\fn\|}{\gamma_\fA \gamma_\fC}\frac{1}{(1-\mu)(1-\nu)}.
$$
Next, for $1\leq j\leq N$, one has
$$
 \sigma_j^{{H}} (\bM^{-1}\bC_\mu \bN^{-1}\bA_\nu - \bI)  = {\sigma_j^{{H}} (\bM^{-1} \bK_{\mu,\nu}) }\leq  \frac{1}{\gamma_\fM}\sigma_j(\fK_{\mu,\nu}),
 $$
with $\bK_{\mu,\nu} = \bC_\mu \bN^{-1} \bA - \bM$ as in \eqref{eq:Kmatrix}. Therefore, we obtain
\begin{align*}
\Theta_k &\leq \frac{\n{(\bP_\mu\bA_\nu)^{-1}}_H}{k} \sum_{j=1}^k \sigma_j^{{H}}(\bM^{-1} \bK_{\mu,\nu}) \leq \frac{\|\fn\|}{\gamma_\fC\gamma_\fA \gamma_\fM} \frac{ \overline{\sigma}_k (\fK_{\mu,\nu})}{(1-\mu)(1-\nu)}.
\end{align*}
The second bound in \eqref{eq:Extended_Theta0} and \eqref{eq:Extended_Thetap} follows by the same arguments as in the proof of \Cref{thm:superlinear_sf}.
\end{proof}

\Cref{thm:superMN} describes precisely the residual convergence behavior of GMRES for \CA$^p_{\mu,\nu}$ for $p\geq0$. In particular, \eqref{eq:Extended_Thetap} shows that the Euclidean GMRES converges super-linearly, up to a $\KL$-term as observed experimentally for the electric field integral equation on screens in \cite{hiptmair2020preconditioning}.
\begin{corollary}[$h$-Asymptotics]\label{remark:Asymptotics} Consider \CA$_{\mu,\nu}$ in \eqref{eq:linear_CAmn}, for $\mu \to 0$ and $\nu\to 0$ as $h\to0$. Additionally, let us suppose that (i) the finite dimensional subspaces are dense in their function space, satisfying the approximability property \cite[Definition 2.14]{ern2004theory}; and, (ii) the forms in \CA~have a uniform discrete inf-sup condition with respect to $h$. Then, for vanishing $h$, the following statements hold:
\begin{enumerate}
\item[(i)] $\|u-u_h\|_X\to 0$ in \Cref{lemma:strang};
\item[(ii)] $\K_\star$ in \Cref{thm:opprecond}, and subsequently $\K_{\star,\mu,\nu}$ \Cref{thm:strang} remain bounded ($h$-independence);
\item[(iii)] Under \Cref{ass:FOVmn}, the residual $\Theta_k^{(m)}$ in \eqref{eq:linear_CAmn} remains bounded ($h$-independent linear convergence);
\item[(iv)] For \CA$^p_{\mu,\nu}$ in \eqref{eq:CAmnp}, $p\geq0$, the residual $\Theta_k \to 0$ as $k\to \infty$ in \Cref{thm:superMN} ($h$-independent super-linear convergence).
\end{enumerate}
\end{corollary}
\begin{remark}

\Cref{thm:superMN} requires the operator $\fK_{\mu,\nu}$ to be compact so as to ensure application of \Cref{prop:classic_superlinear}. Recent results by Bletcha \cite{blechta2021stability} allow to consider a more general \Cref{prop:classic_superlinear} with $\fA : H\to H$ of the form $\fA = \fQ +\fK$, with $\fQ$ a bounded invertible operator and $\fK$ compact. The latter could allow to relax the compactness for $\fK_{\mu,\nu}$ and to analyze \CA$_{\mu,\nu}$ as a general bounded perturbation of \CA$^p$. This will be investigated elsewhere.
\end{remark}

\subsection{Elliptic Case}
We give further insight on the bi-parametric operator preconditioning framework by considering the elliptic case for OP-BG for $X=:H$ and $V$ being Hilbert spaces. To this end, we assume that $\fA$ is $\langle X\rangle$-elliptic and $\fC$ is $\langle V\rangle $-elliptic. Therefore, we have the {ellipticity conditions}
$$
\fa ( u,u) \geq \alpha_\fA \n{u}^2_X\quad \text{and} \quad \fc ( v, v) \geq \alpha_\fC \n{v}^2_{V},
$$
for all $u \in X$ and all $v \in V$. Notice that continuous ellipticity implies a discrete inf-sup condition for conforming discretization spaces, with $\gamma_\fA = \alpha_\fA$ and $\gamma_\fC = \alpha_\fC$, respectively, and allows to apply our previous analysis---without requiring \Cref{ass:FOV}.

For $p\geq0$, problem \CA$^p$ leads to $\fC \fN^{-1} \fA = \fI + \fK$ with $\fK$ compact and self-adjoint. Thus, we introduce the ordered eigenvalues $|\lambda_{i+1}(\fK)|\leq |\lambda_i(\fK)|$ for $i \geq 1$. By \cite[Section 2]{winther1980some}, $|\lambda_i(\fK)|= \sigma_i(\fK)$ and the Carleman class in \eqref{eq:Carleman} simplifies to the Neumann-Schatten class 
\be
\iii{\fK}_p := \left( \sum_{i=1}^\infty |\lambda_i(\fK) |\right)^{1/p}<\infty. 
\ee
As ellipticity allows for more refined bounds, one can examine the use of preconditioned CG solvers \cite[Section 13.1]{steinbach2007numerical}.
\begin{corollary}[Elliptic Case]\label{thm:strang_elliptic} 
Consider \CA$_{\mu,\nu}$ with $\fc_\mu \in \Phi_{h,\mu}(\fc)$, $\fa_\nu \in \Phi_{h,\nu}(\fa)$, such that $\fA_\nu$ is $\langle X\rangle$-elliptic and $\fC_\mu$ is $\langle V \rangle$-elliptic.Then, the continuous and perturbed problems have a unique solution, $u$ and $u_{h,\nu}$, respectively, with the following error bound
\begin{align*}
\n{u - u_{h,\nu}}_X& \leq \inf_{w_h\in X_h} \left(\frac{{\K_\fA}}{1-\nu} \|u - w_h\|_X +\frac{\nu}{1-\nu} \n{w_h}_X\right)
+ \frac{\nu}{\gamma_{\fA}(1-\nu)}\n{b_{{h}}}_{Y_h'}\\
&{\leq  \left(\frac{{\K_\fA^{2}}}{1-\nu}\right)\inf_{w_h\in X_h} \n{u - w_h}_X + \frac{2\nu}{\gamma_{\fA}(1-\nu)}\n{b_{{h}}}_{Y_h'}.}
\end{align*}
with {$\K_\fA$} defined in \eqref{eq:discrete_condition_number}. Furthermore, it holds that 
\be 
\kappa_S (\bP_\mu \bA_\nu) = \kappa_2 (\bP_\mu \bA_\nu) \leq \K_{\star,\mu,\nu},
\ee 
with $\K_{\star,\mu,\nu}$ in \eqref{eq:k_mu_nu}. Therefore, for $\bx_0 \neq \bu_\nu$ and $1\leq k \leq N$, the $k$-th iterate $\bx_k$ of CG with an error $\mathbf{e}_k:= \bx_k - \bu_\nu$ is bounded in the $A_\nu$-norm as
\be\label{eq:cg_bound_linear}
 \Theta_k^\textup{CG}  := \left(\frac{\n{\mathbf{e}_k}_{A_\nu}}{\n{\mathbf{e}_0}_{A_\nu}}\right)^{\frac{1}{k}} \leq 2^\frac{1}{k} \left( 1 - \frac{2}{\sqrt{\K_{\star,\mu,\nu}} +1 }\right).
\ee
Finally, consider \CA$_{\mu,\nu}^p$ for $p\geq 0$. It holds that
\be 
\begin{split}
\Theta_k^\textup{CG} &\leq \frac{2\|\fn\|}{\gamma_\fC \gamma_\fA\gamma_\fM}\frac{1}{(1-\mu)(1-\nu)}\cdot  \frac{1}{k} \sum_{j=1}^k |\lambda_j(\fK_{\mu,\nu})|
\end{split}
\ee
and, if $p>0$, one retrieves
\be 
\begin{split}
\Theta_k^\textup{CG} \leq \frac{2\|\fn\|}{\gamma_\fC \gamma_\fA \gamma_\fM} \frac{\iii{\fK_{\mu,\nu}}_p}{(1-\mu)(1-\nu)}k^{-\frac{1}{p}} .
\end{split}
\ee

\end{corollary}
\begin{proof}By the ellipticity hypothesis on the sesqui-linear form $\fa$, \Cref{lemmaCea} is replaced by the Lax-Milgram lemma \cite[Section 2.1.2]{ern2004theory}, providing the sharper quasi-optimality constant ${\K_\fA}$. Since the resulting system is Hermitian positive definite, the spectral and Euclidean condition numbers coincide. Next, we set $\varkappa:= \kappa_S(\bP_\mu\bA_\nu)$ and introduce the linear bound for the preconditioned CG with respect to the condition number \cite[Theorem 1.8]{kurics2010operator}:
\be 
\Theta_k^\textup{CG} \leq 2^\frac{1}{k} \left(\frac{ \sqrt{\varkappa} -1 }{\sqrt{\varkappa} +1}\right).
\ee
Observe that
$$
\left(\frac{ \sqrt{\varkappa} -1 }{\sqrt{\varkappa} +1}\right) =\left( 1 - \frac{2}{\sqrt{\varkappa}+1}\right)\leq \left( 1 - \frac{2}{\sqrt{\K_{\star,\mu,\nu}}+1}\right),
$$
leading to \eqref{eq:cg_bound_linear}. Since, \CA$^p_{\mu,\nu}$ entails a self-adjoint compact perturbation $\fK_{\mu,\nu}:= \fC_\mu \fN^{-1}\fA_\nu - \fI$, one has an ordered eigenvalue decomposition, and the application of super-linear result for CG \cite[Theorem 1.9]{kurics2010operator}:
$$
\Theta_k^\textup{CG} \leq 2 \n{(\bP_\mu \bA_\nu)^{-1}}_H \left(\frac{1}{k}\sum_{j=1}^k |\lambda_j (\bM^{-1}\bK_{\mu,\nu})|\right).
$$
Finally, one can show that (cf.~proof of \Cref{thm:superMN}):
$$
\n{\bA_\nu^{-1} \bP_\mu^{-1}}_H \leq\frac{\|\fn\|}{\gamma_\fC \gamma_\fA}\frac{1}{(1-\mu)(1-\nu)} \quad \text{and} \quad |\lambda_j (\bM^{-1} \bK_{\mu,\nu})| \leq \frac{1}{\gamma_\fM} |\lambda_j (\fK_{\mu,\nu})|,$$
proving the final result.
\end{proof}
\section{Conclusion}\label{sec:concl}
For general Petrov-Galerkin methods, we considered their operator preconditioning and introduced the novel bi-parametric framework. Several results were derived including bounds in Euclidean norm for the convergence of iterative solvers when preconditioning, with GMRES as a reference. These results pave the way toward new paradigms for preconditioning, as they allow to craft robust preconditioners, better understand the efficiency of existing ones and relate them to experimental results. We see direct applications in a variety of research areas including wave propagation problems \cite{gander2015applying}, singular perturbation theory \cite[Section 3]{axelsson2009equivalent}, fast numerical methods \cite{bebendorf2008hierarchical,bebendorf2009recompression} and iterative solvers \cite{saad1986gmres}. 

Future work avenues we foresee are: further analysis of second-kind Fredholm integral equations, with applications to acoustics and electromagnetics; deep learning of preconditioners for GMRES, and wavenumber asymptotic analysis for preconditioners. Also, we mention two promising research areas: (i) extension of \CA$_{\mu,\nu}^p$ to bounded perturbations of \CA$^p$ via \cite{blechta2021stability}; and (ii) characterization of Carleman class for compact operators using elliptic regularity theorems \cite{bessoud2006q}.

\section*{Acknowledgements}
The authors thank the support of Fondecyt Regular 1171491.\\
\bibliography{references}

\end{document}